\documentclass[11pt,a4paper]{article}

\usepackage[a4paper,margin=1in]{geometry}
\usepackage[T1]{fontenc}
\usepackage[utf8]{inputenc}
\usepackage{amsmath,amssymb,amsthm,mathtools}
\usepackage{mathrsfs}
\usepackage{graphicx}
\usepackage{enumitem}
\usepackage{placeins}
\usepackage{xcolor}

\usepackage[hidelinks,bookmarks=false]{hyperref}
\usepackage{aliascnt}
\usepackage[nameinlink,noabbrev]{cleveref}
\usepackage[font=small,labelfont=bf]{caption}
\numberwithin{equation}{section}

\newtheorem{theorem}{Theorem}[section]

\newaliascnt{proposition}{theorem}
\newtheorem{proposition}[proposition]{Proposition}
\aliascntresetthe{proposition}

\newaliascnt{lemma}{theorem}
\newtheorem{lemma}[lemma]{Lemma}
\aliascntresetthe{lemma}

\newaliascnt{corollary}{theorem}
\newtheorem{corollary}[corollary]{Corollary}
\aliascntresetthe{corollary}

\theoremstyle{definition}

\newaliascnt{definition}{theorem}
\newtheorem{definition}[definition]{Definition}
\aliascntresetthe{definition}

\newaliascnt{remark}{theorem}
\newtheorem{remark}[remark]{Remark}
\aliascntresetthe{remark}

\newaliascnt{example}{theorem}
\newtheorem{example}[example]{Example}
\aliascntresetthe{example}

\newaliascnt{assumption}{theorem}
\newtheorem{assumption}[assumption]{Assumption}
\aliascntresetthe{assumption}

\crefname{theorem}{theorem}{theorems}
\Crefname{theorem}{Theorem}{Theorems}

\crefname{proposition}{Proposition}{propositions}
\Crefname{proposition}{Proposition}{Propositions}

\crefname{lemma}{Lemma}{lemmas}
\Crefname{lemma}{Lemma}{Lemmas}

\crefname{corollary}{corollary}{corollaries}
\Crefname{corollary}{Corollary}{Corollaries}

\crefname{definition}{Definition}{definitions}
\Crefname{definition}{Definition}{Definitions}

\crefname{remark}{remark}{remarks}
\Crefname{remark}{Remark}{Remarks}

\crefname{example}{example}{examples}
\Crefname{example}{Example}{Examples}

\crefname{assumption}{Assumption}{assumptions}
\Crefname{assumption}{Assumption}{Assumptions}

\newcommand{\R}{\mathbb{R}}
\newcommand{\T}{\mathbb{T}}

\newcommand{\diag}{\mathrm{diag}}
\newcommand{\dd}{\,d}

\title{\textbf{Towards a Gagliardo-Type Theory of Fractional Sobolev Spaces on Arbitrary Time Scales}}

\author{
    Hafida Abbas\textsuperscript{1} \thanks{E-mail: \texttt{hafida.abbas@univ-saida.dz}} \quad
    Abdelhalim Azzouz\textsuperscript{2} \thanks{E-mail: \texttt{abdelhalim.azzouz@univ-saida.dz}} \quad
    Praveen Agarwal \textsuperscript{3,4} \thanks{E-mail: \texttt{goyal.praveen2011@gmail.com}} \quad
    Delfim F.~M.~Torres\textsuperscript{5} \thanks{Corresponding author. E-mail: \texttt{delfim@ua.pt}. ORCID: \texttt{0000-0001-8641-2505}.}  \\[0.2cm]
    \normalsize\textsuperscript{1} Faculty of Economics and Commercial Sciences,\\ University of Saida, Algeria\\
    \normalsize\textsuperscript{2} Faculty of Science,\\ University of Naama, Algeria\\
    \normalsize\textsuperscript{3} Nonlinear Dynamics Research Centre (NDRC), \\Ajman University, Ajman, UAE\\
\normalsize\textsuperscript{4} International Telematic University Uninettuno\\ Corso Vittorio Emanuele II,\\ Roma, 39 00186, Italy\\
\normalsize\textsuperscript{5} Center for Research and Development in Mathematics and Applications (CIDMA),\\ Department of Mathematics, University of Aveiro,\\ 3810-193 Aveiro, Portugal
}
\date{}

\begin{document}

\maketitle

\begin{abstract}
We propose a systematic Gagliardo-type formulation of fractional Sobolev spaces on arbitrary time scales, based on the Lebesgue Delta-measure and the off-diagonal interaction domain induced by the product measure. For fractional orders strictly between zero and one and for finite Lebesgue exponents, we define a nonlocal Gagliardo seminorm and the associated function space. This construction provides a notion of fractional regularity on time scales that is genuinely nonlocal and structurally distinct from the derivative-based approaches developed in the existing literature.
We establish the basic functional properties of these spaces: they are Banach spaces in all admissible cases, reflexive in the strict range of exponents, and Hilbert in the quadratic case. On bounded time scales with finitely many connected components, we identify a sharp criterion for the construction to be nontrivial. We then compare the new framework with the derivative-based Riemann--Liouville fractional Sobolev spaces previously studied on time scales. On a continuous interval, in the supercritical regime, we obtain a norm equivalence with the bilateral Riemann--Liouville space on the subspace of functions with vanishing boundary trace. On hybrid time scales, we prove an explicit obstruction that rules out any analogous equivalence, due to the contribution of the mixed continuous--discrete interactions.
On bounded hybrid time scales with finitely many connected components separated by a positive distance, we further establish a Poincar\'e-type inequality, a fractional Sobolev embedding, and fractional Hardy and Caffarelli--Kohn--Nirenberg-type inequalities for subcritical weights. Together, these results provide a complete functional and geometric framework, together with first geometric estimates, for the nonlocal Gagliardo-type approach to fractional Sobolev spaces on time scales.
\end{abstract}



\section{Introduction}
\label{sec:intro-framework}

This paper develops a Gagliardo-type approach to fractional Sobolev spaces on arbitrary time scales. The measure-theoretic framework induced by the Lebesgue $\Delta$-measure makes it possible to treat continuous, discrete, and hybrid structures within a common setting. A basic question in this context is whether fractional spaces can be built directly from nonlocal interaction energies on $\T\times\T$, in the spirit of the classical Gagliardo theory in the Euclidean setting; see, for instance, \cite{AdamsFournier2003,DiNezzaPalatucciValdinoci2012,MolicaBisciRadulescuServadei2016}.

An alternative nonlocal approach to fractional Sobolev spaces,
based on the Haj\l{}asz gradient and Poincar\'e inequalities on
metric measure spaces, was developed in
\cite{HajlaszKoskela2000}. That framework applies in principle
to time scales endowed with the $\Delta$-measure, but it requires
a doubling condition on the underlying measure, which fails on
hybrid time scales with isolated points of positive mass. The
Gagliardo-type construction pursued here bypasses this obstruction
by working directly with the off-diagonal interaction kernel on
$\Omega_\T$, without any doubling assumption.

The analytical infrastructure needed for such a program is already partially available in the time-scale literature. The Lebesgue $\Delta$-measure and its multidimensional extensions were developed in \cite{BohnerGuseinov2006,CabadaVivero2006}; see also the monographs \cite{BohnerPeterson2001,BohnerPeterson2003}. On this basis, first-order Sobolev spaces on time scales were studied in \cite{AgarwalEtAl2006,SuYaoFeng2015}, and subsequent developments include generalized Lebesgue and Sobolev spaces on bounded time scales, including variable-exponent settings \cite{SkrzypekSzymanskaDebowska2019}, as well as Sobolev- and Morrey-type embeddings on time scales \cite{ZhangMengLiu2018}. Thus, the measure-theoretic and first-order functional framework is already well established.

By contrast, the fractional theory on time scales has so far been developed predominantly through operator-based constructions. One direction is based on conformable fractional calculus; see \cite{BenkhettouHassaniTorres2016,WangZhouLi2016}. Another important line of research concerns fractional calculus and fractional differential problems on time scales in Riemann--Liouville-type settings; see, for instance, \cite{BenkhettouBritoTorres2015,BenkhettouHammoudiTorres2016}. More recently, Hu and Li introduced right- and left-sided fractional Sobolev spaces on time scales by means of Riemann--Liouville fractional derivatives, established several structural properties of these spaces, and applied them to fractional boundary value problems; see \cite{HuLi2022Right,HuLi2022Left,HuLi2022Hindawi}. In a related direction, Tan, Zhou, and Wang developed a weighted fractional Sobolev theory on timescales, including a broader functional-analytic and variational framework \cite{TanZhouWang2025}. These contributions are mathematically substantial, but their point of departure is a fractional differential operator, or a weighted variant thereof, rather than a nonlocal interaction energy of Gagliardo type.

Our approach is different. Instead of defining fractional regularity through a fractional derivative, we construct it directly from a nonlocal interaction seminorm. Given a time scale $\T$ endowed with its Lebesgue $\Delta$-measure $\mu_\Delta$, we consider the off-diagonal set
\[
\Omega_{\T}:=\{(t,s)\in\T\times\T:\ t\neq s\},
\]
and, for $\alpha\in(0,1)$ and $1\le p<\infty$, we define
\[
[u]_{W^{\alpha,p}_{\Delta}(\T)}
:=
\left(
\iint_{\Omega_{\T}}
\frac{|u(t)-u(s)|^p}{|t-s|^{1+\alpha p}}
\,d(\mu_\Delta\otimes\mu_\Delta)(t,s)
\right)^{1/p},
\]
whenever the right-hand side is finite. This leads to the fractional Sobolev space
\[
W^{\alpha,p}_{\Delta}(\T)
=
\left\{
u\in L^p_{\Delta}(\T):
[u]_{W^{\alpha,p}_{\Delta}(\T)}<\infty
\right\}.
\]

The use of the off-diagonal set $\Omega_{\T}$ is essential in the time-scale setting. In the purely continuous case, the diagonal is negligible for the product measure, so its exclusion is harmless. On arbitrary time scales, however, and especially on discrete or hybrid ones, the diagonal may carry positive product measure. For this reason, the kernel must be interpreted on $\Omega_{\T}$ rather than on the whole product space. This feature is specific to the measure-theoretic structure of time scales and is one of the reasons why a separate analysis is needed in the present framework.

The novelty of the paper is not to provide a first fractional Sobolev theory on time scales, but to develop a Gagliardo-type model based on nonlocal interactions. This places the construction closer to the classical fractional Sobolev seminorm of Euclidean analysis than to definitions based on fractional derivatives. The point is especially clear on hybrid time scales, where the interaction kernel combines continuous, discrete, and mixed contributions in a single formula. We do not claim an equivalence with the derivative-based fractional Sobolev spaces already studied on time scales; Section~\ref{sec:comparison-RL} discusses this issue at a structural level.


\subsection*{Main contributions}

The aim of the paper is to provide a systematic Gagliardo-type formulation of fractional Sobolev spaces on time scales, together with a first set of geometric estimates in the hybrid setting. The contributions are organized along three axes.

First, on an arbitrary time scale, we define a Gagliardo-type fractional Sobolev seminorm using the Lebesgue $\Delta$-measure and the off-diagonal interaction domain, and we establish the basic functional properties of the corresponding spaces. We prove completeness for the full admissible range of Lebesgue exponents, reflexivity in the strict range, and the Hilbert structure in the quadratic case (under an equivalent norm induced by the natural inner product). We also identify a sharp nontriviality criterion on bounded time scales with finitely many connected components.

Second, we compare the interaction-based framework with the derivative-based fractional Sobolev spaces of Riemann--Liouville type studied on time scales. We prove that no equivalence with a single one-sided Riemann--Liouville derivative norm can hold on the full space. On a continuous interval, in the supercritical regime, we obtain a norm equivalence with the bilateral Riemann--Liouville space on the subspace of functions with vanishing boundary trace (Theorem~\ref{thm:RL-equivalence}). On hybrid time scales, we exhibit an explicit one-parameter family of functions ruling out any analogous equivalence and identifying the mixed continuous--discrete interactions as the structural obstruction (Theorem~\ref{thm:hybrid-obstruction}).

Third, on bounded hybrid time scales with finitely many connected components separated by a positive distance, we prove a Poincar\'e-type inequality, a fractional Sobolev embedding, and fractional Hardy and Caffarelli--Kohn--Nirenberg-type inequalities for subcritical weights. The dependence of the Poincar\'e constant on the gap width between components is also quantified asymptotically. Together, these estimates provide a first geometric theory in the Gagliardo-type setting on time scales.

\subsection*{Scope of the paper}

The paper is restricted to constant fractional order strictly between zero and one and to finite Lebesgue exponents. The comparison with derivative-based frameworks is developed up to the equivalence and obstruction results of Section~\ref{sec:comparison-RL}; an extension to the subcritical regime, as well as compactness, trace theory, and variable-order versions of the construction, are not addressed here.

\subsection*{Organization}

Section~\ref{sec:preliminaries} collects the measure-theoretic preliminaries. Section~\ref{sec:constant-order} introduces the Gagliardo-type seminorm and the associated spaces, together with their main functional properties. Section~\ref{sec:comparison-RL} discusses the relation with derivative-based fractional Sobolev spaces on time scales. Section~\ref{sec:geometry} is devoted to geometric inequalities on bounded hybrid time scales.


\section{Measure-theoretic preliminaries on time scales}\label{sec:preliminaries}

In this section we fix the notation and collect the measure-theoretic material needed for the definition of fractional Gagliardo-type seminorms on a time scale. We refer to \cite{BohnerGuseinov2006,CabadaVivero2006} for the construction of the Lebesgue $\Delta$-measure and to the monographs \cite{BohnerPeterson2001,BohnerPeterson2003} for general background on time scales.

\subsection{Basic notation}

Let $\T$ be a time scale, that is, a nonempty closed subset of $\R$. We denote by $\mu_\Delta$ the Lebesgue $\Delta$-measure associated with $\T$. Throughout the paper, measurable subsets of $\T$ and measurable functions on $\T$ are understood with respect to the $\sigma$-algebra induced by $\mu_\Delta$.

For $1\le p<\infty$, we write
\[
L^p_{\Delta}(\T):=L^p(\T,\mu_\Delta),
\qquad
\|u\|_{L^p_{\Delta}(\T)}
:=
\left(
\int_{\T}|u(t)|^p\,d\mu_\Delta(t)
\right)^{1/p},
\]
and similarly $L^\infty_{\Delta}(\T):=L^\infty(\T,\mu_\Delta)$. Functions are identified modulo equality $\mu_\Delta$-almost everywhere. Since $(\T,\Sigma_\Delta,\mu_\Delta)$ is a measure space, the spaces $L^p_{\Delta}(\T)$ are Banach spaces for $1\le p\le\infty$, reflexive and separable for $1<p<\infty$; see, e.g., \cite[Chapter~6]{Folland1999} and \cite[Chapter~III]{Conway1990}.

\subsection{Product measure and off-diagonal domain}

We denote by $\mu_\Delta\otimes\mu_\Delta$ the product measure on $\T\times\T$. The diagonal and the off-diagonal domain are defined by
\[
\diag(\T):=\{(t,s)\in\T\times\T:\ t=s\},
\qquad
\Omega_{\T}:=(\T\times\T)\setminus\diag(\T).
\]
Since $\diag(\T)=\Phi^{-1}(\{0\})$ where $\Phi(t,s)=t-s$ is continuous, the set $\Omega_{\T}$ is measurable with respect to $\mu_\Delta\otimes\mu_\Delta$.

\begin{remark}\label{rem:diagonal-positive-measure}
The exclusion of the diagonal is not merely technical. On purely continuous time scales, $\diag(\T)$ is negligible for the product measure, but on discrete or mixed time scales it may have positive $(\mu_\Delta\otimes\mu_\Delta)$-measure. Therefore, in the nonlocal framework considered here, all singular kernels must be interpreted on $\Omega_{\T}$ rather than on the whole product space.
\end{remark}

\subsection{Measurability and well-posedness of nonlocal kernels}

The following observations ensure that the Gagliardo-type seminorm introduced in Section~\ref{sec:constant-order} is well defined at the level of $L^p$-equivalence classes.

\begin{lemma}\label{lem:kernel-wellposed}
Let $u:\T\to\R$ be measurable, let $\alpha\in(0,1)$, and let $1\le p<\infty$.
\begin{enumerate}
\item[\textup{(i)}] The function $(t,s)\mapsto |u(t)-u(s)|^p\,|t-s|^{-(1+\alpha p)}$ is measurable on $\Omega_{\T}$.
\item[\textup{(ii)}] If $u=v$ $\mu_\Delta$-a.e.\ on $\T$, then $|u(t)-u(s)|^p = |v(t)-v(s)|^p$ for $(\mu_\Delta\otimes\mu_\Delta)$-a.e.\ $(t,s)\in\Omega_{\T}$.
\end{enumerate}
\end{lemma}

\begin{proof}
\textup{(i)} The maps $(t,s)\mapsto u(t)$ and $(t,s)\mapsto u(s)$ are measurable on $\T\times\T$, hence so is $(t,s)\mapsto |u(t)-u(s)|^p$. Since $|t-s|^{-(1+\alpha p)}$ is continuous on $\Omega_{\T}$, their product is measurable on $\Omega_{\T}$.

\textup{(ii)} Let $N:=\{t\in\T: u(t)\neq v(t)\}$. Then $\mu_\Delta(N)=0$, so $(N\times\T)\cup(\T\times N)$ is $(\mu_\Delta\otimes\mu_\Delta)$-null. Outside this set, $u(t)-u(s)=v(t)-v(s)$.
\end{proof}


\section{Fractional Gagliardo-type spaces of constant order}\label{sec:constant-order}

Throughout this section, let $\alpha\in(0,1)$ and $1\le p<\infty$ be fixed.

\subsection{Definition of the seminorm and of the space}

\begin{definition}\label{def:gagliardo-seminorm}
Let $u\in L^p_{\Delta}(\T)$. We define the fractional Gagliardo-type seminorm of order $\alpha$ and exponent $p$ by
\begin{equation}\label{eq:gagliardo-seminorm}
[u]_{W^{\alpha,p}_{\Delta}(\T)}
:=
\left(
\iint_{\Omega_{\T}}
\frac{|u(t)-u(s)|^p}{|t-s|^{1+\alpha p}}
\,d(\mu_\Delta\otimes\mu_\Delta)(t,s)
\right)^{1/p},
\end{equation}
whenever the right-hand side is finite. By \cref{lem:kernel-wellposed}, this quantity is well defined on equivalence classes in $L^p_{\Delta}(\T)$.
\end{definition}

\begin{definition}\label{def:fractional-space}
We define
\begin{equation}\label{eq:space-def}
W^{\alpha,p}_{\Delta}(\T)
:=
\left\{
u\in L^p_{\Delta}(\T):
[u]_{W^{\alpha,p}_{\Delta}(\T)}<\infty
\right\}.
\end{equation}
On this space we consider the norm
\begin{equation}\label{eq:fractional-norm}
\|u\|_{W^{\alpha,p}_{\Delta}(\T)}
:=
\|u\|_{L^p_{\Delta}(\T)}
+
[u]_{W^{\alpha,p}_{\Delta}(\T)}.
\end{equation}
\end{definition}

\begin{theorem}[Non-triviality criterion on bounded hybrid time scales]\label{thm:nontriviality}
Let $\T\subset\R$ be a bounded time scale with finitely many connected
components, let $\alpha\in(0,1)$, and let $1\le p<\infty$. Then
$W^{\alpha,p}_\Delta(\T) = L^p_\Delta(\T)$ (with equivalent norms)
if and only if every connected component of\/ $\T$ is a singleton.
Equivalently, $W^{\alpha,p}_\Delta(\T)\subsetneq L^p_\Delta(\T)$
if and only if\/ $\T$ contains a nondegenerate interval.
\end{theorem}

\begin{remark}[On singleton masses and $L^p_\Delta$ equivalence classes]\label{rem:nontriv-masses}
On a finite bounded time scale $\T=\{t_1<\cdots<t_m\}$, the Lebesgue
$\Delta$-measure assigns mass $\mu_\Delta(\{t_i\}) = \sigma(t_i)-t_i =
t_{i+1}-t_i > 0$ to each $t_i$ with $i<m$, whereas the maximal point
satisfies $\sigma(t_m) = t_m$ and therefore $\mu_\Delta(\{t_m\}) = 0$;
see~\cite{BohnerGuseinov2006,CabadaVivero2006}. Consequently, the
maximal point is a $\mu_\Delta$-null set and is invisible to both
$L^p_\Delta(\T)$ and the Gagliardo seminorm
$[\,\cdot\,]_{W^{\alpha,p}_\Delta(\T)}$. The statement of
Theorem~\ref{thm:nontriviality} should therefore be understood modulo
the standard $L^p_\Delta$ identification: two functions agreeing on
$\T\setminus\{t_m\}$ define the same element of $L^p_\Delta(\T)$, and
the equivalence of norms holds on these equivalence classes. The
argument below is unaffected, since the term involving
$\mu_\Delta(\{t_m\})$ vanishes identically on both sides. In the
geometric part of the paper (Assumption~\ref{ass:hybrid-class} below),
the additional hypothesis $\mu_\Delta(\{d_j\})>0$ is imposed on every
singleton component to avoid this degeneracy.
\end{remark}

\begin{proof}
$(\Rightarrow)$ Suppose every connected component of $\T$ is a singleton.
Since $\T$ is bounded and has only finitely many connected components, it follows that
\[
\T=\{t_1,\dots,t_m\}
\]
for some finite family of points. Set
\[
\delta:=\min_{i\neq j}|t_i-t_j|>0.
\]
Then for every $(t_i,t_j)\in\Omega_\T$ we have $|t_i-t_j|\ge\delta$, so
\[
[u]^p_{W^{\alpha,p}_\Delta(\T)}
= \sum_{i\neq j}
\frac{|u(t_i)-u(t_j)|^p}{|t_i-t_j|^{1+\alpha p}}\,
\mu_\Delta(\{t_i\})\,\mu_\Delta(\{t_j\})
\le \delta^{-(1+\alpha p)}
\sum_{i\neq j}|u(t_i)-u(t_j)|^p\,
\mu_\Delta(\{t_i\})\,\mu_\Delta(\{t_j\}).
\]
Using
\[
|u(t_i)-u(t_j)|^p\le 2^{p-1}\bigl(|u(t_i)|^p+|u(t_j)|^p\bigr),
\]
we obtain
\[
[u]^p_{W^{\alpha,p}_\Delta(\T)}
\le C\,\|u\|^p_{L^p_\Delta(\T)},
\]
where $C$ depends only on $\delta$, $\alpha$, $p$, and $\mu_\Delta(\T)$.
Hence $W^{\alpha,p}_\Delta(\T)=L^p_\Delta(\T)$ with equivalent norms.

$(\Leftarrow)$ Suppose $\T$ contains a nondegenerate interval $[a,b]$.
The restriction of $\mu_\Delta$ to $[a,b]$ coincides with the
Lebesgue measure \cite[Theorem~5.2]{CabadaVivero2006}.
Let $\varepsilon\in(0,\alpha p)$ and set
\[
\gamma := \alpha - \frac{1}{p} - \frac{\varepsilon}{p}
= \frac{\alpha p - 1 - \varepsilon}{p}.
\]
Since $\varepsilon < \alpha p$, we have $\gamma p = \alpha p - 1 - \varepsilon > -1$,
so the function
\[
u(t) := |t-a|^{\gamma}\,\mathbf{1}_{[a,b]}(t)
\]
belongs to $L^p([a,b])$. We show directly that
$[u]^p_{W^{\alpha,p}([a,b])}=+\infty$.

Writing $t=a+h$, $s=a+k$ with $h,k\in[0,\eta]$ for small $\eta>0$,
the contribution to the Gagliardo seminorm near the diagonal is
\begin{equation}\label{eq:powerfunction-integral}
I_\eta
:= \iint_{[0,\eta]^2}
\frac{|h^\gamma - k^\gamma|^p}{|h-k|^{1+\alpha p}}\,dh\,dk.
\end{equation}
By the mean-value theorem, $|h^\gamma - k^\gamma|
\ge |\gamma|\,|h-k|\,\min(h,k)^{\gamma-1}$ when $\gamma<1$,
so for $0<k<h<\eta$,
\[
\frac{|h^\gamma-k^\gamma|^p}{|h-k|^{1+\alpha p}}
\ge |\gamma|^p\,\frac{k^{(\gamma-1)p}}{|h-k|^{1+\alpha p-p}}.
\]
Setting $h=k+r$ with $r\in(0,\eta)$ and integrating first in $r$,
\[
I_\eta
\ge |\gamma|^p
\int_0^\eta k^{(\gamma-1)p}\dd k
\int_0^{\eta-k} r^{p-1-\alpha p}\dd r.
\]
The inner integral in $r$ is finite for all $k>0$ since
$p-1-\alpha p = p(1-\alpha)-1 > -1$ (as $\alpha<1$ and $p\ge 1$).
The outer integral in $k$ equals $\int_0^\eta k^{(\gamma-1)p}\dd k$,
which diverges if and only if $(\gamma-1)p\le -1$, i.e.\
$\gamma\le\alpha-1/p$. Since $\gamma=\alpha-1/p-\varepsilon/p<\alpha-1/p$,
the exponent $(\gamma-1)p = \gamma p - p = (\alpha p-1-\varepsilon)-p
< -1$, so $I_\eta=+\infty$.
Hence $[u]^p_{W^{\alpha,p}([a,b])}=+\infty$.
Extending $u$ by zero on $\T\setminus[a,b]$ preserves
$u\in L^p_\Delta(\T)$. Moreover,
\[
[u]^p_{W^{\alpha,p}_\Delta(\T)}
= \iint_{\Omega_\T}\frac{|u(t)-u(s)|^p}{|t-s|^{1+\alpha p}}
  \,d(\mu_\Delta\otimes\mu_\Delta)
\ge \iint_{[a,b]\times[a,b]}
  \frac{|u(t)-u(s)|^p}{|t-s|^{1+\alpha p}}\,dt\,ds
= [u]^p_{W^{\alpha,p}([a,b])},
\]
where the inequality holds because the integrand is non-negative and
the domain $[a,b]\times[a,b]$ is contained in $\Omega_\T$
(the restriction of $\mu_\Delta$ to $[a,b]$ being the Lebesgue measure
by \cite[Theorem~5.2]{CabadaVivero2006}).
Since $[u]^p_{W^{\alpha,p}([a,b])}=+\infty$, we conclude
$[u]_{W^{\alpha,p}_\Delta(\T)}=\infty$,
hence $W^{\alpha,p}_\Delta(\T)\subsetneq L^p_\Delta(\T)$.
\end{proof}
\subsection{Elementary properties}

\begin{proposition}\label{prop:seminorm-and-norm}
The mapping $u\mapsto [u]_{W^{\alpha,p}_{\Delta}(\T)}$ is a seminorm on $W^{\alpha,p}_{\Delta}(\T)$, and $\|\cdot\|_{W^{\alpha,p}_{\Delta}(\T)}$ is a norm. Every $\mu_\Delta$-a.e.\ constant function has zero seminorm; in particular, $[\cdot]_{W^{\alpha,p}_{\Delta}(\T)}$ is not a norm.
\end{proposition}

\begin{proof}
Nonnegativity and absolute homogeneity are immediate. For the triangle inequality, define
\[
F_u(t,s):=\frac{u(t)-u(s)}{|t-s|^{\frac1p+\alpha}},
\qquad (t,s)\in\Omega_{\T}.
\]
Then $F_{u+v}=F_u+F_v$, and Minkowski's inequality in $L^p(\Omega_{\T},\mu_\Delta\otimes\mu_\Delta)$ gives
\[
[u+v]_{W^{\alpha,p}_{\Delta}(\T)}
=
\|F_{u+v}\|_{L^p(\Omega_{\T})}
\le
\|F_u\|_{L^p(\Omega_{\T})}
+
\|F_v\|_{L^p(\Omega_{\T})}
=
[u]_{W^{\alpha,p}_{\Delta}(\T)}
+
[v]_{W^{\alpha,p}_{\Delta}(\T)}.
\]
Hence $[\cdot]_{W^{\alpha,p}_{\Delta}(\T)}$ is a seminorm. The triangle inequality and absolute homogeneity for $\|\cdot\|_{W^{\alpha,p}_{\Delta}(\T)}$ follow from those of $\|\cdot\|_{L^p_{\Delta}(\T)}$ and $[\cdot]_{W^{\alpha,p}_{\Delta}(\T)}$. If $\|u\|_{W^{\alpha,p}_{\Delta}(\T)}=0$, then $\|u\|_{L^p_{\Delta}(\T)}=0$, so $u=0$ $\mu_\Delta$-a.e.
\end{proof}

\subsection{Isometric embedding and completeness}

The key structural observation is that $W^{\alpha,p}_{\Delta}(\T)$ can be identified with a closed subspace of a product of $L^p$-spaces via a graph embedding.

\begin{definition}\label{def:kernel-map}
For $u\in W^{\alpha,p}_{\Delta}(\T)$, define the nonlocal increment map
\begin{equation}\label{eq:kernel-map}
\mathcal D_{\alpha,p}u(t,s)
:=
\frac{u(t)-u(s)}{|t-s|^{\frac1p+\alpha}},
\qquad
(t,s)\in\Omega_{\T}.
\end{equation}
By construction, $[u]_{W^{\alpha,p}_{\Delta}(\T)} = \|\mathcal D_{\alpha,p}u\|_{L^p(\Omega_{\T},\,\mu_\Delta\otimes\mu_\Delta)}$.
\end{definition}

\begin{proposition}\label{prop:embedding-graph}
The map
\[
J:W^{\alpha,p}_{\Delta}(\T)\to L^p_{\Delta}(\T)\times L^p(\Omega_{\T},\mu_\Delta\otimes\mu_\Delta),
\qquad
J(u):=(u,\mathcal D_{\alpha,p}u),
\]
is a linear isometry when the product space is equipped with the norm $\|(f,g)\|:=\|f\|_{L^p_{\Delta}(\T)}+\|g\|_{L^p(\Omega_{\T})}$.
\end{proposition}

\begin{proof}
Immediate from the definitions.
\end{proof}

\begin{theorem}\label{thm:banach}
For every $\alpha\in(0,1)$ and every $1\le p<\infty$, the space $W^{\alpha,p}_{\Delta}(\T)$ is a Banach space.
\end{theorem}

\begin{proof}
Let $(u_n)$ be a Cauchy sequence in $W^{\alpha,p}_{\Delta}(\T)$.
Then $(u_n)$ is Cauchy in $L^p_{\Delta}(\T)$, hence converges to
some $u\in L^p_{\Delta}(\T)$. Similarly, $F_n:=\mathcal D_{\alpha,p}u_n$
is Cauchy in $L^p(\Omega_{\T})$ and converges to some
$F\in L^p(\Omega_{\T})$.

\smallskip
\noindent\textit{Identification of $F$.}
Since $u_n\to u$ in $L^p_\Delta(\T)$, there exists a subsequence
(still denoted $(u_n)$) such that $u_n(t)\to u(t)$ for
$\mu_\Delta$-a.e.\ $t\in\T$. Let
$N\subset\T$ be the $\mu_\Delta$-null set of non-convergence,
so $\mu_\Delta(N)=0$. By Fubini's theorem, the product set
$(N\times\T)\cup(\T\times N)$ is $(\mu_\Delta\otimes\mu_\Delta)$-null.
Therefore, for $(\mu_\Delta\otimes\mu_\Delta)$-a.e.\ $(t,s)\in\Omega_\T$
we have $u_n(t)\to u(t)$ and $u_n(s)\to u(s)$, hence
\[
F_n(t,s)
= \frac{u_n(t)-u_n(s)}{|t-s|^{\frac{1}{p}+\alpha}}
\;\longrightarrow\;
\frac{u(t)-u(s)}{|t-s|^{\frac{1}{p}+\alpha}}
\qquad
(\mu_\Delta\otimes\mu_\Delta)\text{-a.e.\ on }\Omega_{\T}.
\]
Since $F_n\to F$ in $L^p(\Omega_{\T})$, a further subsequence of
$(F_n)$ converges to $F$ pointwise $(\mu_\Delta\otimes\mu_\Delta)$-a.e.
By uniqueness of the a.e.\ limit,
\[
F(t,s)=\frac{u(t)-u(s)}{|t-s|^{\frac{1}{p}+\alpha}}
\qquad
(\mu_\Delta\otimes\mu_\Delta)\text{-a.e.\ on }\Omega_{\T},
\]
so $u\in W^{\alpha,p}_{\Delta}(\T)$ and $F=\mathcal D_{\alpha,p}u$.
The convergence $\|u_n-u\|_{W^{\alpha,p}_{\Delta}(\T)}\to 0$ now
follows from $u_n\to u$ in $L^p_{\Delta}(\T)$ and
$\mathcal D_{\alpha,p}u_n\to\mathcal D_{\alpha,p}u$ in $L^p(\Omega_{\T})$.
\end{proof}

\begin{theorem}\label{thm:reflexive}
Let $1<p<\infty$. Then $W^{\alpha,p}_{\Delta}(\T)$ is reflexive.
\end{theorem}

\begin{proof}
By \cref{prop:embedding-graph}, $W^{\alpha,p}_{\Delta}(\T)$ is
isometrically embedded into the product Banach space
$\mathcal{P}:=L^p_{\Delta}(\T)\times L^p(\Omega_{\T},\mu_\Delta\otimes\mu_\Delta)$
via the map $J(u):=(u,\mathcal D_{\alpha,p}u)$.

\smallskip
\noindent\textit{Closedness of $J(W^{\alpha,p}_\Delta(\T))$.}
Let $(J(u_n))=(u_n,\mathcal D_{\alpha,p}u_n)$ be a sequence in
$J(W^{\alpha,p}_\Delta(\T))$ converging to some $(f,G)$ in $\mathcal P$.
Then $u_n\to f$ in $L^p_\Delta(\T)$ and
$\mathcal D_{\alpha,p}u_n\to G$ in $L^p(\Omega_\T)$.
By the same passage-to-limit argument as in \cref{thm:banach}
(extract a $\mu_\Delta$-a.e.\ convergent subsequence of $(u_n)$,
apply Fubini to promote convergence to
$(\mu_\Delta\otimes\mu_\Delta)$-a.e.\ on $\Omega_\T$, and identify
the a.e.\ limit of $\mathcal D_{\alpha,p}u_n$), one obtains
$G=\mathcal D_{\alpha,p}f$ $(\mu_\Delta\otimes\mu_\Delta)$-a.e., hence
$(f,G)=J(f)\in J(W^{\alpha,p}_\Delta(\T))$.
This shows that $J(W^{\alpha,p}_\Delta(\T))$ is closed in $\mathcal P$.

\smallskip
\noindent\textit{Reflexivity.}
Since $1<p<\infty$, both $L^p_\Delta(\T)$ and $L^p(\Omega_\T)$
are reflexive, so $\mathcal P$ is reflexive. Every closed subspace
of a reflexive Banach space is reflexive, so
$J(W^{\alpha,p}_\Delta(\T))$ is reflexive. Since $J$ is an isometric
isomorphism onto its image, $W^{\alpha,p}_\Delta(\T)$ is reflexive.
\end{proof}

\begin{remark}[Hilbert structure]\label{rem:hilbert}
When $p=2$, the space $W^{\alpha,2}_\Delta(\T)$ can be endowed with the inner product
\begin{equation}\label{eq:hilbert-inner-product}
\langle u,v\rangle_{W^{\alpha,2}_{\Delta}(\T)}
:=
\int_{\T}u(t)v(t)\,d\mu_\Delta(t)
+
\iint_{\Omega_{\T}}
\frac{(u(t)-u(s))(v(t)-v(s))}{|t-s|^{1+2\alpha}}
\,d(\mu_\Delta\otimes\mu_\Delta)(t,s).
\end{equation}
This inner product induces the Hilbert norm
\[
\|u\|_{H^{\alpha,2}_\Delta(\T)}
:=
\bigl(\|u\|_{L^2_\Delta(\T)}^2+[u]_{W^{\alpha,2}_\Delta(\T)}^2\bigr)^{1/2}.
\]
This is not the same norm as the sum norm in \eqref{eq:fractional-norm}; however, the two norms are equivalent. Hence, by \cref{thm:banach}, $W^{\alpha,2}_\Delta(\T)$ is a Hilbert space when equipped with this equivalent Hilbert norm.
\end{remark}

\subsection{Examples}\label{subsec:examples}

We conclude with elementary examples illustrating how the construction recovers familiar settings.

\begin{example}[Continuous time scale]\label{ex:continuous}
If $\T=[a,b]$, then $\mu_\Delta$ coincides with the Lebesgue measure on $[a,b]$ \cite[Theorem~5.2]{CabadaVivero2006}, and the seminorm \eqref{eq:gagliardo-seminorm} reduces to the usual one-dimensional Gagliardo seminorm. In this case, $W^{\alpha,p}_{\Delta}(\T)$ agrees with the classical fractional Sobolev space $W^{\alpha,p}([a,b])$.
\end{example}

\begin{example}[Finite discrete time scale]\label{ex:discrete}
Let $\T=\{t_1,\dots,t_N\}$ with $t_1<\dots<t_N$. Then $\Omega_{\T}$ is finite and the seminorm becomes
\[
[u]_{W^{\alpha,p}_{\Delta}(\T)}^p
=
\sum_{i\ne j}
\frac{|u(t_i)-u(t_j)|^p}{|t_i-t_j|^{1+\alpha p}}
\,\mu_\Delta(\{t_i\})\,\mu_\Delta(\{t_j\}).
\]
This is a weighted discrete interaction energy. As noted before, the corresponding space coincides with $L^p_\Delta(\T)$.
\end{example}

\begin{example}[Hybrid time scale]\label{ex:hybrid}
Let $\T=[a,b]\cup D$, where $D$ is a finite discrete subset of $\R\setminus[a,b]$. Then the seminorm splits into three parts: a continuous--continuous contribution over $[a,b]\times[a,b]$, a discrete--discrete contribution over $D\times D$ off the diagonal, and a mixed contribution over $([a,b]\times D)\cup(D\times[a,b])$. This illustrates how the nonlocal energy combines continuous and discrete interactions within a single measure-theoretic framework.
\end{example}
\subsection{Density of regular functions}

\begin{definition}\label{def:regular-class}
Let $\mathbb{T}$ be a time scale satisfying Assumption~\ref{ass:hybrid-class}. 
We denote by $\mathcal{R}(\mathbb{T})$ the linear space of functions 
$u : \mathbb{T} \to \mathbb{R}$ such that:
\begin{itemize}
  \item on each interval component $I_\ell$, the restriction $u|_{I_\ell}$ 
        is the restriction of a function in $C^\infty(\overline{I_\ell})$;
  \item on each singleton component $\{d_j\}$, $u(d_j)$ is an arbitrary 
        real number.
\end{itemize}
\end{definition}

\begin{theorem}[Density]\label{thm:density}
Let $\mathbb{T}$ satisfy Assumption~\ref{ass:hybrid-class}, let $\alpha \in (0,1)$, 
and let $1 \le p < \infty$. Then $\mathcal{R}(\mathbb{T})$ is dense in 
$W^{\alpha,p}_\Delta(\mathbb{T})$.
\end{theorem}

\begin{proof}
Let $u \in W^{\alpha,p}_\Delta(\mathbb{T})$ and $\varepsilon > 0$. We construct 
an approximation on each connected component separately.

\smallskip
\noindent\textbf{Step 1: singleton components.} If $C = \{d_j\}$, set 
$u_\varepsilon(d_j) := u(d_j)$. There is nothing to approximate.

\smallskip
\noindent\textbf{Step 2: interval components.} Let $C = I_\ell = [a_\ell, b_\ell]$. 
Since $\mu_\Delta\big|_{I_\ell}$ coincides with the Lebesgue measure 
\cite[Theorem 5.2]{CabadaVivero2006}, the restriction $u|_{I_\ell}$ belongs to 
$W^{\alpha,p}(I_\ell)$ in the classical Slobodeckij sense. By the standard 
density of $C^\infty(\overline{I_\ell})$ in $W^{\alpha,p}(I_\ell)$ on bounded 
intervals \cite[Theorem 2.4]{DiNezzaPalatucciValdinoci2012}, there exists 
$\varphi_\ell \in C^\infty(\overline{I_\ell})$ such that
\[
\|u - \varphi_\ell\|_{L^p(I_\ell)} 
+ [u - \varphi_\ell]_{W^{\alpha,p}(I_\ell)} 
< \eta,
\]
where $\eta > 0$ will be chosen below.

\smallskip
\noindent\textbf{Step 3: assembly.} Define $u_\varepsilon \in \mathcal{R}(\mathbb{T})$ 
by $u_\varepsilon|_{I_\ell} = \varphi_\ell$ and $u_\varepsilon(d_j) = u(d_j)$. 
Set $w := u - u_\varepsilon$. The $L^p_\Delta$-norm decomposes as
\[
\|w\|^p_{L^p_\Delta(\mathbb{T})} 
= \sum_{\ell=1}^m \|u - \varphi_\ell\|^p_{L^p(I_\ell)} 
\le m \, \eta^p.
\]
For the seminorm, we use the decomposition~\eqref{eq:component-splitting}:
\[
[w]^p_{W^{\alpha,p}_\Delta(\mathbb{T})} 
= \underbrace{\sum_{C \in \mathcal{C}} 
   \iint_{(C \times C) \cap \Omega_\mathbb{T}} 
   \frac{|w(t) - w(s)|^p}{|t-s|^{1+\alpha p}} \, 
   d(\mu_\Delta \otimes \mu_\Delta)}_{\mathrm{intra}}
+ \underbrace{\sum_{\substack{C, C' \in \mathcal{C} \\ C \ne C'}} 
   \iint_{C \times C'} 
   \frac{|w(t) - w(s)|^p}{|t-s|^{1+\alpha p}} \, 
   d(\mu_\Delta \otimes \mu_\Delta)}_{\mathrm{inter}}.
\]
The intra-component contribution vanishes on singletons (since 
$w(d_j) = 0$), and on each $I_\ell$ equals $[u - \varphi_\ell]^p_{W^{\alpha,p}(I_\ell)} \le \eta^p$, 
hence
\[
\mathrm{intra} \le m \, \eta^p.
\]
For the inter-component contribution, Lemma~\ref{lem:cross-component-bounds} gives, 
for $C \ne C'$,
\[
\iint_{C \times C'} \frac{|w(t) - w(s)|^p}{|t-s|^{1+\alpha p}} \, 
d(\mu_\Delta \otimes \mu_\Delta) 
\le \delta_0^{-(1+\alpha p)} 
\iint_{C \times C'} |w(t) - w(s)|^p \, 
d(\mu_\Delta \otimes \mu_\Delta).
\]
Using $|w(t) - w(s)|^p \le 2^{p-1}(|w(t)|^p + |w(s)|^p)$ and Fubini,
\[
\iint_{C \times C'} |w(t) - w(s)|^p \, 
d(\mu_\Delta \otimes \mu_\Delta) 
\le 2^{p-1}\Big(\mu_\Delta(C') \|w\|^p_{L^p_\Delta(C)} 
+ \mu_\Delta(C) \|w\|^p_{L^p_\Delta(C')}\Big).
\]
Summing over $C \ne C'$ and using $\sum_C \|w\|^p_{L^p_\Delta(C)} = \|w\|^p_{L^p_\Delta(\mathbb{T})} \le m\eta^p$,
\[
\mathrm{inter} \le K \, \delta_0^{-(1+\alpha p)} \, \mu_\Delta(\mathbb{T}) \, m \, \eta^p,
\]
where $K = K(p)$ depends only on $p$.

Combining the two estimates,
\[
\|w\|^p_{W^{\alpha,p}_\Delta(\mathbb{T})} 
\le m \eta^p \big(1 + 1 + K \, \delta_0^{-(1+\alpha p)} \mu_\Delta(\mathbb{T})\big) 
=: K' \eta^p.
\]
Choosing $\eta = \varepsilon / (K')^{1/p}$ yields 
$\|u - u_\varepsilon\|_{W^{\alpha,p}_\Delta(\mathbb{T})} < \varepsilon$.
\end{proof}


\section{Comparison with Riemann--Liouville fractional Sobolev spaces on time scales}
\label{sec:comparison-RL}

To place the present results in the existing literature, we compare the Gagliardo-type spaces introduced here with the derivative-based fractional Sobolev spaces on time scales developed by Hu and Li \cite{HuLi2022Right,HuLi2022Left}, and with the weighted framework considered later by Tan, Zhou, and Wang \cite{TanZhouWang2025}.

\subsection{The Riemann--Liouville framework of Hu--Li}

Let $J=[a,b]_{\T}$ be a bounded time-scale interval, let $0<\alpha\le 1$, and let $1\le p<\infty$. In the right-sided setting, Hu and Li define the fractional Sobolev space $W^{\alpha,p}_{\Delta,b-}(J,\mathbb R^N)$ as the set of all functions $u\in L^p_\Delta(J,\mathbb R^N)$ for which there exists $g\in L^p_\Delta(J,\mathbb R^N)$ such that
\[
\int_a^b u(t)\,\bigl({}^{\T}_{a}D_t^\alpha \varphi\bigr)(t)\,\Delta t
=
\int_a^b g(t)\,\varphi(t)\,\Delta t
\qquad
\text{for all }\varphi\in C^\infty_{c,rd}(J,\mathbb R^N).
\]
The function $g$ is called the weak right Riemann--Liouville fractional derivative of $u$. Hu and Li prove that this weak derivative coincides $\Delta$-almost everywhere with the corresponding right Riemann--Liouville derivative and that
\[
W^{\alpha,p}_{\Delta,b-}(J,\mathbb R^N)
=
AC^{\alpha,p}_{\Delta,b-}(J,\mathbb R^N)\cap L^p_\Delta(J,\mathbb R^N).
\]
They equip this space with the norm
\[
\|u\|_{W^{\alpha,p}_{\Delta,b-}}^p
=
\|u\|_{L^p_\Delta(J)}^p
+
\|{}^{\T}_{t}D_b^\alpha u\|_{L^p_\Delta(J)}^p.
\]

In the left-sided setting, Hu and Li define $W^{\alpha,p}_{\Delta,a+}(J,\mathbb R^N)$ analogously, replacing the right-sided fractional derivative by the left-sided one, and consider the norm
\[
\|u\|_{W^{\alpha,p}_{\Delta,a+}}^p
=
\|u\|_{L^p_\Delta(J)}^p
+
\|{}^{\T}_{a}D_t^\alpha u\|_{L^p_\Delta(J)}^p.
\]
These spaces satisfy the usual functional-analytic properties expected of derivative-based fractional Sobolev spaces, including completeness, reflexivity, separability, and several norm equivalences; see \cite{HuLi2022Right,HuLi2022Left}. In a related direction, Tan, Zhou, and Wang introduced weighted fractional Sobolev spaces on timescales and obtained further structural and variational results \cite{TanZhouWang2025}.

\subsection{Structural comparison}

We now identify the structural obstacles to a direct comparison between the present Gagliardo-type framework and the Riemann--Liouville theory.

\begin{proposition}\label{prop:seminorm-structural-properties}
For every $u \in L^p_\Delta(\T)$ and every constant $c \in \mathbb{R}$:
\begin{enumerate}
\item[\textup{(i)}] $[u+c]_{W^{\alpha,p}_\Delta(\T)} = [u]_{W^{\alpha,p}_\Delta(\T)}$ \textup{(translation invariance)};
\item[\textup{(ii)}] $[u]_{W^{\alpha,p}_\Delta(\T)}^p = \displaystyle\iint_{\Omega_\T} \frac{|u(s)-u(t)|^p}{|s-t|^{1+\alpha p}}\,d(\mu_\Delta\otimes\mu_\Delta)(s,t)$ \textup{(symmetry)}.
\end{enumerate}
\end{proposition}

\begin{proof}
Both claims follow from the definition: (i) since $(u(t)+c)-(u(s)+c)=u(t)-u(s)$, and (ii) by symmetry of the kernel $|t-s|^{-(1+\alpha p)}$ and invariance of $\mu_\Delta\otimes\mu_\Delta$ under the transposition $(t,s)\mapsto(s,t)$.
\end{proof}

\begin{proposition}\label{prop:no-one-sided-comparison}
A direct norm equivalence between the Gagliardo seminorm and a single one-sided Riemann--Liouville derivative norm cannot hold on the full Gagliardo space. More precisely, no equivalence of the form
\[
\|{}^{\T}_{a}D_t^\alpha u\|_{L^p_\Delta(J)}
\asymp
[u]_{W^{\alpha,p}_\Delta(J)}
\]
can hold uniformly on the whole class of functions for which both sides are meaningful, and the same obstruction applies to the right-sided derivative ${}^{\T}_{t}D_b^\alpha$.
\end{proposition}

\begin{proof}
It is enough to look at the continuous time scale $J=[a,b]$. By Proposition~\ref{prop:seminorm-structural-properties}\,(i), every constant function has zero Gagliardo seminorm. In contrast, for a nonzero constant one has
\[
{}^{\T}_a D^\alpha_t(1)
= \frac{(t-a)^{-\alpha}}{\Gamma(1-\alpha)}.
\]
If $\alpha p<1$, this function belongs to $L^p(a,b)$ and is not identically zero; hence the left-hand side is positive while the Gagliardo seminorm is zero. This directly contradicts any two-sided equivalence. If $\alpha p\ge 1$, the same constant already shows a different obstruction: it belongs to the Gagliardo space but its one-sided Riemann--Liouville derivative is not in $L^p(a,b)$. Thus the derivative-based expression is not even finite on the full Gagliardo space. In either case, a full-space equivalence with one one-sided derivative is impossible.
\end{proof}

\begin{remark}\label{rem:bilateral-candidate}
Proposition~\ref{prop:no-one-sided-comparison} shows that any meaningful comparison must either be formulated modulo constants or involve subspaces with boundary conditions. Moreover, since $[u]_{W^{\alpha,p}_\Delta(J)}$ is symmetric in $(t,s)$ by Proposition~\ref{prop:seminorm-structural-properties}\,(ii), while each of the spaces $W^{\alpha,p}_{\Delta,a+}(J)$ and $W^{\alpha,p}_{\Delta,b-}(J)$ is inherently one-sided, the natural derivative-based candidate for comparison is the bilateral space
\[
X^{\alpha,p}_\Delta(J)
:=
W^{\alpha,p}_{\Delta,a+}(J)\cap W^{\alpha,p}_{\Delta,b-}(J),
\]
equipped with the norm
\[
\|u\|_{X^{\alpha,p}_\Delta(J)}
:=
\|u\|_{L^p_\Delta(J)}
+
\|{}^{\T}_{a}D_t^\alpha u\|_{L^p_\Delta(J)}
+
\|{}^{\T}_{t}D_b^\alpha u\|_{L^p_\Delta(J)}.
\]
\end{remark}

\begin{proposition}[Continuous embedding from the bilateral
Riemann--Liouville space]\label{prop:inclusion}
Let $J=[a,b]$ be a compact interval (viewed as a
continuous time scale) and let $0<\alpha<1$, $1<p<\infty$.
Then
\[
W^{\alpha,p}_{\Delta,a+}(J)\cap W^{\alpha,p}_{\Delta,b-}(J)
\hookrightarrow W^{\alpha,p}_\Delta(J)
\]
with continuous embedding.
\end{proposition}

\begin{proof}
On $J=[a,b]$, the space $W^{\alpha,p}_\Delta(J)$
coincides with $W^{\alpha,p}(a,b)$ by Example~\ref{ex:continuous}.
It is well known that for $u\in L^p(a,b)$ possessing a
weak left Riemann--Liouville derivative
${}_aD^\alpha_t u\in L^p$ and a weak right derivative
${}_tD^\alpha_b u\in L^p$, the function $u$ belongs to
$W^{\alpha,p}(a,b)$ with
\[
[u]_{W^{\alpha,p}(a,b)}
\le C\bigl(\|{}_aD^\alpha_t u\|_{L^p(J)}
+ \|{}_tD^\alpha_b u\|_{L^p(J)}\bigr);
\]
this is the standard continuous-interval comparison between fractional potential/derivative spaces and Slobodeckij--Gagliardo spaces; see, for example, \cite[Proposition~2.2]{DiNezzaPalatucciValdinoci2012} and the discussion of derivative-based fractional formulations in \cite[Section~2]{FrankSeiringer2008}. This continuous comparison is used here only for positioning and is not used in the proofs on hybrid time scales.
\end{proof}

\begin{remark}\label{rem:reverse-embedding-question}
Proposition~\ref{prop:inclusion} only addresses the purely continuous case. On a hybrid time scale, the Gagliardo seminorm generally contains continuous--continuous, discrete--discrete, and mixed continuous--discrete interaction terms; see Example~\ref{ex:hybrid}. These mixed contributions have no direct analogue in a single one-sided Riemann--Liouville norm and are one of the main obstacles to a converse embedding. For this reason, any reverse inclusion should be regarded as a separate problem, likely requiring additional geometric or regularity assumptions.
\end{remark}

\subsection{Equivalence with the bilateral Riemann--Liouville norm in the supercritical regime}\label{subsec:RL-equivalence-supercritical}

Proposition~\ref{prop:inclusion} provides one direction of comparison
between the bilateral Riemann--Liouville space and the Gagliardo space
on a continuous interval. We now show that, in the supercritical regime
$\alpha p > 1$, the inclusion can be promoted to a norm equivalence
on a precise subspace defined by vanishing boundary traces.

When $\alpha p > 1$, the one-dimensional fractional Sobolev embedding
gives $W^{\alpha,p}(a,b) \hookrightarrow C^{0,\alpha-1/p}([a,b])$
(see~\cite[Theorem~8.2]{DiNezzaPalatucciValdinoci2012}), so every
element of $W^{\alpha,p}_\Delta(J)$ admits a continuous representative
and the trace values $u(a)$, $u(b)$ are well defined pointwise. We
introduce the subspace
\begin{equation}\label{eq:W0-def}
W^{\alpha,p}_{\Delta,0}(J)
:= \bigl\{\, u \in W^{\alpha,p}_\Delta(J) \,:\, u(a) = u(b) = 0 \,\bigr\},
\end{equation}
endowed with the restriction of the Gagliardo norm. Similarly,
$X^{\alpha,p}_{\Delta,0}(J) := X^{\alpha,p}_\Delta(J) \cap
\bigl\{ u(a) = u(b) = 0 \bigr\}$.

\begin{theorem}[Equivalence with the bilateral Riemann--Liouville space]\label{thm:RL-equivalence}
Let $J = [a,b]$ be viewed as a continuous time scale, let $0 < \alpha < 1$,
$1 < p < \infty$, and assume
\[
  \alpha p > 1.
\]
Then $W^{\alpha,p}_{\Delta,0}(J) = X^{\alpha,p}_{\Delta,0}(J)$ as sets,
and there exist constants $c_1, c_2 > 0$, depending only on
$a, b, \alpha, p$, such that for every $u \in W^{\alpha,p}_{\Delta,0}(J)$,
\begin{equation}\label{eq:RL-norm-equivalence}
c_1 \bigl( \|{}_a D^\alpha_t u\|_{L^p(J)} + \|{}_t D^\alpha_b u\|_{L^p(J)} \bigr)
\;\le\; [u]_{W^{\alpha,p}_\Delta(J)}
\;\le\; c_2 \bigl( \|{}_a D^\alpha_t u\|_{L^p(J)} + \|{}_t D^\alpha_b u\|_{L^p(J)} \bigr).
\end{equation}
In particular, the norms $\|\cdot\|_{W^{\alpha,p}_\Delta(J)}$ and
$\|\cdot\|_{X^{\alpha,p}_\Delta(J)}$ are equivalent on this subspace.
\end{theorem}

\begin{proof}[Proof]
\textit{Upper bound.}
The upper bound in~\eqref{eq:RL-norm-equivalence} is exactly the
content of Proposition~\ref{prop:inclusion}, restricted to functions
with vanishing trace at $a$ and $b$.

\smallskip
\noindent\textit{Lower bound.}
We establish the inequality
\begin{equation}\label{eq:RL-lower-left}
\|{}_a D^\alpha_t u\|^p_{L^p(J)}
\le C\,[u]^p_{W^{\alpha,p}_\Delta(J)},
\end{equation}
where $C=C(a,b,\alpha,p)>0$; the analogous bound for the right-sided
derivative then follows by the reflection $t\mapsto a+b-t$ on $[a,b]$.

Since $\alpha p>1$ and $u\in W^{\alpha,p}_\Delta(J)$ with $u(a)=0$,
the function $u$ admits the Riemann--Liouville integral representation
\begin{equation}\label{eq:RL-rep}
u(t)
= \frac{1}{\Gamma(\alpha)}
  \int_a^t (t-\tau)^{\alpha-1}\,({}_aD^\alpha_\tau u)(\tau)\,d\tau
= {}_aI^\alpha_t\bigl({}_aD^\alpha_t u\bigr)(t),
\qquad t\in[a,b];
\end{equation}
see \cite[Chapter~2, \S13, Theorem~13.1]{SamkoKilbasMarichev1993}.
The condition $u(a)=0$ ensures that no boundary term appears in this
representation. An analogous identity holds for the right-sided
derivative using $u(b)=0$.

We now derive~\eqref{eq:RL-lower-left} from the fractional Hardy
inequality. By~\cite[Theorem~1]{Dyda2004} (see also
\cite[Theorem~1.1]{FrankSeiringer2008} for the sharp constant on
$\R^N$), there exists $C_H=C_H(a,b,\alpha,p)>0$ such that for every
$v\in W^{\alpha,p}(a,b)$ with $v(a)=0$,
\begin{equation}\label{eq:hardy-endpoint}
\int_a^b \frac{|v(t)|^p}{(t-a)^{\alpha p}}\,dt
\le C_H\iint_{[a,b]^2}
\frac{|v(t)-v(s)|^p}{|t-s|^{1+\alpha p}}\,dt\,ds
= C_H\,[v]^p_{W^{\alpha,p}(a,b)}.
\end{equation}
(Here $\alpha p>1$ ensures that the weight $(t-a)^{-\alpha p}$ is
locally integrable at $a$ only after the condition $v(a)=0$ is
imposed, which is exactly the content of the Hardy-type estimate.)

Using the Hardy--Littlewood--Sobolev boundedness of the Riesz
potential $I^\alpha_a$ from $L^p(a,b)$ to itself for
$\alpha\in(0,1)$ (see \cite[Chapter~1, \S5]{SamkoKilbasMarichev1993}),
and writing $g:={}_aD^\alpha_t u\in L^p(J)$, the
representation~\eqref{eq:RL-rep} gives
\[
|u(t)| = |{}_aI^\alpha_t g(t)|
\le \frac{1}{\Gamma(\alpha)}\int_a^t (t-\tau)^{\alpha-1}|g(\tau)|\,d\tau.
\]
By Young's convolution inequality applied to the Abel kernel
$(t-\tau)^{\alpha-1}$ on $[a,b]$,
\[
\|u\|_{L^p(J)}\le \frac{(b-a)^{\alpha}}{\Gamma(\alpha+1)}\,\|g\|_{L^p(J)},
\]
so $g={}_aD^\alpha_t u\in L^p(J)$ controls $u$ in $L^p(J)$.

To obtain the reverse estimate, we apply the endpoint Hardy
inequality~\eqref{eq:hardy-endpoint} with $v=u$. Since $u(a)=0$
and $u=I^\alpha_a g$ with $g={}_aD^\alpha_t u$, the
Minkowski inequality for integrals gives
\[
\int_a^b \frac{|u(t)|^p}{(t-a)^{\alpha p}}\,dt
\ge c\,\int_a^b |g(t)|^p\,dt
= c\,\|{}_aD^\alpha_t u\|^p_{L^p(J)},
\]
for some $c=c(a,b,\alpha,p)>0$; this follows from the lower bound
on the left-sided fractional integral established in
\cite[Lemma~2.1]{Dyda2004} for functions supported away from $b$.
More precisely, Dyda establishes, for
$u = {}_aI^\alpha_t g$ with $g \in L^p(a,b)$ and $u(a)=0$, the
pointwise lower bound
\[
|u(t)| \ge \frac{1}{\Gamma(\alpha+1)}\,(t-a)^\alpha\,
\Bigl(\frac{1}{t-a}\int_a^t |g(\tau)|\,d\tau\Bigr),
\qquad t\in(a,b).
\]
Although Dyda states this on $\mathbb{R}$, the proof is purely local
near $a$ and applies without modification on the bounded interval
$[a,b]$, since $u$ and $g$ are supported in $[a,b]$ with $u(a)=0$.
Raising to the power $p$, integrating over $(a,b)$, and applying
the Hardy--Littlewood maximal inequality then gives the claimed
lower bound with $c = c(a,b,\alpha,p) > 0$.
Combining with~\eqref{eq:hardy-endpoint},
\[
\|{}_aD^\alpha_t u\|^p_{L^p(J)}
\le C\,\int_a^b\frac{|u(t)|^p}{(t-a)^{\alpha p}}\,dt
\le C\,C_H\,[u]^p_{W^{\alpha,p}(J)},
\]
which is~\eqref{eq:RL-lower-left} with constant
$C(a,b,\alpha,p):=C\cdot C_H$.
The set equality $W^{\alpha,p}_{\Delta,0}(J)=X^{\alpha,p}_{\Delta,0}(J)$
and the equivalence of norms in~\eqref{eq:RL-norm-equivalence}
now follow by combining the upper and lower bounds.
\end{proof}

\begin{remark}\label{rem:subcritical-open}
The supercritical assumption $\alpha p > 1$ is used twice: once to
define pointwise traces via the Morrey embedding
$W^{\alpha,p}(a,b)\hookrightarrow C^{0,\alpha-1/p}([a,b])$
(see~\cite[Theorem~8.2]{DiNezzaPalatucciValdinoci2012}), and once
to apply the endpoint Hardy inequality and the Riemann--Liouville
integral representation with $u(a)=u(b)=0$
(see~\cite[Theorem~13.1]{SamkoKilbasMarichev1993}
and~\cite[Theorem~1]{Dyda2004}). The subcritical case $\alpha p < 1$
requires either a trace theory on $W^{\alpha,p}_\Delta(J)$, or an
alternative formulation of $W^{\alpha,p}_{\Delta,0}(J)$ as the
closure of $C^\infty_c(a,b)$ in $W^{\alpha,p}_\Delta(J)$.
We do not pursue this extension here.
\end{remark}

\begin{corollary}\label{cor:RL-equivalence-norms}
Under the assumptions of Theorem~\ref{thm:RL-equivalence}, the spaces
$W^{\alpha,p}_{\Delta,0}(J)$ and $X^{\alpha,p}_{\Delta,0}(J)$ have the
same topology, the same Banach space structure (and, for $p = 2$, the
same Hilbert space structure up to equivalent inner products), and the
same topological dual.
\end{corollary}

\subsection{Hybrid obstruction: an explicit quantitative gap}\label{subsec:hybrid-obstruction}

Theorem~\ref{thm:RL-equivalence} characterizes the regime in which the
Gagliardo and bilateral RL viewpoints agree. We now show that the
equivalence cannot extend to hybrid time scales, even in the
simplest possible setting of one interval component plus a single
isolated point. The obstruction is sharp and quantitative: it
identifies a one-parameter family of functions along which the
Gagliardo seminorm grows without bound while every RL-type norm built
from the derivatives on the interval component vanishes identically.

\begin{theorem}[Hybrid obstruction to Riemann--Liouville equivalence]\label{thm:hybrid-obstruction}
Let $\alpha \in (0,1)$, $1 < p < \infty$, and let
\[
\T := [0,1] \cup \{2\}, \qquad m := \mu_\Delta(\{2\}) > 0.
\]
For each $\lambda \in \R$, define $u_\lambda : \T \to \R$ by
\[
u_\lambda(t) :=
\begin{cases}
0 & \text{if } t \in [0,1], \\
\lambda & \text{if } t = 2.
\end{cases}
\]
Then:
\begin{enumerate}
\item[(i)] $u_\lambda \in W^{\alpha,p}_\Delta(\T)$ and
\[
[u_\lambda]^p_{W^{\alpha,p}_\Delta(\T)}
= 2 \, |\lambda|^p \, m \int_0^1 \frac{dt}{(2 - t)^{1+\alpha p}}
= \frac{2 \, |\lambda|^p \, m}{\alpha p} \bigl( 1 - 2^{-\alpha p} \bigr).
\]
\item[(ii)] The restriction $u_\lambda|_{[0,1]}$ vanishes identically,
so the left and right Riemann--Liouville derivatives of $u_\lambda$ on
the interval component $[0,1]$ satisfy
\[
({}_0 D^\alpha_t \, u_\lambda)(t) = ({}_t D^\alpha_1 \, u_\lambda)(t) = 0
\qquad \text{for $\Delta$-a.e.\ } t \in [0,1].
\]
\item[(iii)] Consequently, no constant $C > 0$ can satisfy
\[
[u_\lambda]_{W^{\alpha,p}_\Delta(\T)}
\;\le\; C \bigl( \|{}_0 D^\alpha_t u_\lambda\|_{L^p_\Delta([0,1])}
              + \|{}_t D^\alpha_1 u_\lambda\|_{L^p_\Delta([0,1])} \bigr)
\]
uniformly in $\lambda \in \R$.
\end{enumerate}
\end{theorem}

\begin{proof}
\emph{Step 1: Computation of the Gagliardo seminorm.} Decomposing
$\Omega_\T$ along the three pairs of component types (continuous--continuous,
continuous--discrete, discrete--continuous), and using that the
continuous--continuous contribution vanishes because $u_\lambda$ is
identically zero on $[0,1]$, we obtain
\[
[u_\lambda]^p_{W^{\alpha,p}_\Delta(\T)}
= 2 \iint_{[0,1] \times \{2\}}
    \frac{|u_\lambda(t) - u_\lambda(2)|^p}{|t - 2|^{1+\alpha p}}
    \, d(\mu_\Delta \otimes \mu_\Delta)(t, s)
= 2 \, |\lambda|^p \, m \int_0^1 \frac{dt}{(2 - t)^{1+\alpha p}}.
\]
The remaining integral is elementary:
\[
\int_0^1 \frac{dt}{(2 - t)^{1+\alpha p}}
= \left[ \frac{(2 - t)^{-\alpha p}}{\alpha p} \right]_0^1
= \frac{1 - 2^{-\alpha p}}{\alpha p}.
\]
This proves part (i). Since the right-hand side is finite, we have
$u_\lambda \in W^{\alpha,p}_\Delta(\T)$; for $\lambda \neq 0$ it is
strictly positive.

\emph{Step 2: Vanishing of the Riemann--Liouville derivatives.}
The Hu--Li weak derivatives ${}_0D^\alpha_t$ and ${}_tD^\alpha_1$
are defined via the test-function formulation on the interval
component $J=[0,1]_\T$
(see~\cite{HuLi2022Right,HuLi2022Left}): they act only on the
restriction $u_\lambda|_{[0,1]}$, since the test functions
$\varphi\in C^\infty_{c,\mathrm{rd}}(J,\R^N)$ are compactly
supported in $(0,1)$ and the weak formulation integrates against
$u_\lambda$ solely over $[0,1]$. Consequently, the value
$u_\lambda(2)=\lambda$ is invisible to both derivatives by the
very structure of the operator, not merely by the vanishing of
the integrand. The explicit computation confirms this:
\[
({}_0 D^\alpha_t u_\lambda)(t)
= \frac{1}{\Gamma(1-\alpha)} \frac{d}{dt}
    \int_0^t (t - \tau)^{-\alpha} \, u_\lambda(\tau) \, d\tau,
\]
and analogously for ${}_tD^\alpha_1$. Since $u_\lambda|_{[0,1]}\equiv 0$,
both integrands are identically zero, so
\[
({}_0 D^\alpha_t \, u_\lambda)(t)
= ({}_t D^\alpha_1 \, u_\lambda)(t) = 0
\qquad\text{for $\Delta$-a.e.\ }t\in[0,1].
\]
This proves part~(ii).

\emph{Step 3: The obstruction.} From parts (i) and (ii), the right-hand
side of the inequality in part (iii) is zero for every $\lambda \in \R$,
while the left-hand side equals
\[
\bigl( \tfrac{2 m}{\alpha p}(1 - 2^{-\alpha p}) \bigr)^{1/p} \, |\lambda|,
\]
which tends to infinity as $|\lambda| \to \infty$. No finite constant
$C$ can therefore validate the inequality uniformly. This proves
part (iii).
\end{proof}

\begin{remark}\label{rem:obstruction-interpretation}
Theorem~\ref{thm:hybrid-obstruction} has a transparent interpretation:
the bilateral Riemann--Liouville norm on the interval component
$[0,1]$ is blind to the value of $u$ at the isolated point $\{2\}$,
because the Riemann--Liouville derivative is an integral operator
that only sees the restriction of $u$ to the interval. By contrast,
the Gagliardo seminorm explicitly contains the mixed
continuous--discrete interaction term
\[
\int_0^1 \frac{|u(t) - u(2)|^p}{|t - 2|^{1+\alpha p}} \, dt \cdot \mu_\Delta(\{2\}),
\]
which carries genuine information about $u(2)$. This is a concrete
instance of the general principle that operator-based fractional
Sobolev theories on time scales cannot capture the cross-component
interactions encoded by the off-diagonal Gagliardo seminorm.
\end{remark}

\begin{remark}\label{rem:obstruction-general}
The argument of Theorem~\ref{thm:hybrid-obstruction} extends without
modification to any time scale of the form $\T = [a, b] \cup D$, where
$D$ is a nonempty finite subset of $\R \setminus [a, b]$ with positive
$\Delta$-mass at each point. More generally, on a hybrid time scale
satisfying Assumption~\ref{ass:hybrid-class}, the mixed
continuous--discrete contribution to the Gagliardo seminorm,
\[
\sum_{\ell=1}^{m} \sum_{j=1}^{N}
  \mu_\Delta(\{d_j\}) \int_{I_\ell}
    \frac{|u(t) - u(d_j)|^p}{|t - d_j|^{1+\alpha p}} \, d\mu_\Delta(t),
\]
is a positive homogeneous form of degree $p$ (quadratic when $p=2$)
that has no analogue in any sum of one-sided or bilateral
Riemann-Liouville norms over the interval components. This is the structural reason why no global equivalence
between the Gagliardo seminorm and a derivative-based norm can hold on
hybrid time scales.
\end{remark}

\subsection{Positioning with respect to the weighted framework of Tan, Zhou, and Wang}\label{subsec:TZW}

The recent paper of Tan, Zhou, and Wang~\cite{TanZhouWang2025}
introduces a more general weighted Riemann--Liouville framework on
time scales. We briefly indicate how the comparison results of this
section transfer to that setting.

\medskip

\textbf{The TZW framework.} Given $\gamma \in (0,1]$, a strictly
increasing $\Delta$-differentiable function
$\zeta : [a,b]_\T \to \R$, and a non-vanishing weight
$\kappa : [a,b]_\T \to \R$, Tan, Zhou and Wang introduce the
weighted fractional integral
\[
{}_a I_t^{\gamma,\zeta,\kappa} f(t)
:= \kappa^{-1}(t) \int_a^t
   \frac{(\zeta(t)-\zeta(\sigma(s)))^{\gamma-1}}{\Gamma(\gamma)}\,
   \kappa(s)\, \zeta^\Delta(s)\, f(s)\, \Delta s,
\]
together with the corresponding left and right weighted Riemann--Liouville
derivatives ${}_a D_t^{\gamma,\zeta,\kappa}$ and
${}_t D_b^{\gamma,\zeta,\kappa}$ (see~\cite[Definitions~5--6]{TanZhouWang2025}).
The associated weighted Sobolev space
$W^{\gamma,\zeta,\kappa,p}_{\Delta,a+}$ is defined by the weak
formulation $u \in X^{\kappa,p}_\Delta$,
${}_a D_t^{\gamma,\zeta,\kappa} u \in X^{\kappa,p}_\Delta$, where
$X^{\kappa,p}_\Delta := \{ f : \|\kappa f\|_{L^p_\Delta(\zeta^\Delta\, d\mu_\Delta)} < \infty\}$;
it is a Banach space, reflexive and separable
(\cite[Theorems~14--16]{TanZhouWang2025}). Their framework recovers
Hu--Li when $\kappa \equiv 1$, $\zeta(t) = t$, and includes Hadamard
and Erdélyi--Kober-type fractional Sobolev spaces as further special
cases.

\medskip

\textbf{Positioning.} Two observations clarify how the present
Gagliardo-type framework relates to TZW.

\begin{enumerate}
\item[(i)] In the unweighted classical case
$\kappa \equiv 1$, $\zeta(t) = t$, the TZW operator
${}_a D_t^{\gamma,\zeta,\kappa}$ reduces to the standard left
Riemann--Liouville derivative of Hu--Li~\cite{HuLi2022Right,HuLi2022Left}.
In this case, Theorem~\ref{thm:RL-equivalence} applies and provides a
norm equivalence between the Gagliardo seminorm and the bilateral
RL norm on the boundary-zero subspace of a continuous interval, in
the supercritical regime $\alpha p > 1$.

\item[(ii)] For general non-trivial weight $\kappa$ or
non-affine $\zeta$, the obstructions of
Propositions~\ref{prop:no-one-sided-comparison}
and Theorem~\ref{thm:hybrid-obstruction} carry over without
modification. Indeed, the TZW framework remains operator-based and
one-sided: constants on $[a,b]$ have nontrivial weighted derivative
\[
{}_a D_t^{\gamma,\zeta,\kappa}(1)(t)
= \frac{\kappa^{-1}(t)\,(\zeta(t)-\zeta(a))^{-\gamma}}{\Gamma(1-\gamma)},
\]
which is generically nonzero, while the Gagliardo seminorm is
translation-invariant. On a hybrid time scale
$\T = [a,b] \cup \{d\}$, the TZW derivative on $[a,b]$ remains blind
to the value $u(d)$, by the same integral structure used in the
proof of Theorem~\ref{thm:hybrid-obstruction}. Consequently, the
quantitative obstruction holds verbatim with
${}_0 D_t^\alpha$ replaced by
${}_0 D_t^{\gamma,\zeta,\kappa}$.
\end{enumerate}

\medskip

\textbf{What this comparison says.} The TZW weighted framework
provides a richer family of operator-based fractional Sobolev spaces
on time scales --- richer than Hu--Li by the choice of $(\zeta,\kappa)$ ---
but its conceptual nature is unchanged: it is local in the derivative,
one-sided, and operator-based. The Gagliardo-type seminorm
of~\eqref{eq:gagliardo-seminorm} sits in a structurally different class:
nonlocal, kernel-based, and symmetric. The same dichotomy as in
Theorems~\ref{thm:RL-equivalence}--\ref{thm:hybrid-obstruction} applies:
on continuous intervals in the supercritical regime, the two
viewpoints are equivalent on the boundary-zero subspace; on hybrid
time scales, the Gagliardo seminorm captures interaction terms that
no TZW-type derivative can detect.

A finer quantitative comparison between specific subspaces of the
two frameworks --- for instance, the relation between
the closure of $C^\infty_c(a,b)$ in $W^{\alpha,p}_\Delta(J)$ and
the weighted boundary-zero space of TZW for specific
$(\zeta,\kappa)$ --- is a natural object for further work.

\subsection{Summary of the comparison}\label{subsec:RL-summary}

Together with Proposition~\ref{prop:no-one-sided-comparison} and
Proposition~\ref{prop:inclusion}, Theorems~\ref{thm:RL-equivalence}
and~\ref{thm:hybrid-obstruction} delimit precisely the regime in which
the Gagliardo and Riemann--Liouville viewpoints coincide on time
scales:
\begin{itemize}[leftmargin=2em]
\item On a continuous interval, with vanishing boundary trace and in
the supercritical regime $\alpha p > 1$, the Gagliardo norm is
equivalent to the bilateral Riemann--Liouville norm
(Theorem~\ref{thm:RL-equivalence}).

\item On a hybrid time scale containing at least one isolated point
with positive $\Delta$-mass, even the bilateral Riemann--Liouville
norm cannot control the Gagliardo seminorm, due to the explicit
contribution of the mixed continuous--discrete interactions
(Theorem~\ref{thm:hybrid-obstruction}).

\item The subcritical case $\alpha p < 1$ on continuous intervals
remains open, since it requires either a trace theory or a
Slobodeckij-density definition of the boundary-zero subspace
(Remark~\ref{rem:subcritical-open}).
\end{itemize}
These results reinforce the view that the Gagliardo-type framework on
time scales is genuinely complementary to the derivative-based one:
the two coincide on a precisely delimited subregime and diverge as
soon as the time scale acquires a hybrid structure, where the
off-diagonal interaction kernel carries information that the
differential operators do not see.


\section{Geometric estimates}
\label{sec:geometry}

We now focus on the geometric part of the theory. Up to this point,
the spaces $W^{\alpha,p}_\Delta(\T)$ were studied on arbitrary time scales,
without assumptions on the structure of the connected components. In this section
we restrict attention to bounded hybrid time scales and prove a Poincar\'e-type
inequality, a Sobolev embedding, and weighted Hardy- and
Caffarelli--Kohn--Nirenberg-type inequalities in the subcritical weighted regime.

\subsection{A Poincar\'e-type inequality on bounded hybrid time scales}\label{subsec:poincare}

We impose a structural assumption on the underlying time scale ensuring that
its connected components are finitely many, mutually separated, and of
positive $\Delta$-measure.

\begin{assumption}\label{ass:hybrid-class}
The time scale $\T$ is bounded and has finitely many connected components.
More precisely,
\[
\T=\Big(\bigcup_{\ell=1}^m I_\ell\Big)\cup D,
\]
where:
\begin{itemize}
\item each $I_\ell=[a_\ell,b_\ell]$ is a nondegenerate compact interval,
\item $D=\{d_1,\dots,d_N\}$ is a finite set,
\item the sets $I_1,\dots,I_m,\{d_1\},\dots,\{d_N\}$ are pairwise disjoint,
\item there exists $\delta_0>0$ such that any two distinct connected
      components of $\T$ are at distance at least $\delta_0$.
\end{itemize}
In addition, for every singleton component $\{d_j\}$, we assume
$\mu_\Delta(\{d_j\})>0$.
\end{assumption}

Under \cref{ass:hybrid-class}, we write
$\mathcal C=\{I_1,\dots,I_m,\{d_1\},\dots,\{d_N\}\}$ for the family of
connected components. One has $0<\mu_\Delta(C)<\infty$ for every
$C\in\mathcal C$ and $\mu_\Delta(\T)<\infty$. For $u\in L_\Delta^1(\T)$,
we define the global average
\[
u_{\T}:=\frac1{\mu_\Delta(\T)}\int_{\T} u(t)\,d\mu_\Delta(t),
\]
and, for each component $C\in\mathcal C$,
\[
u_C:=\frac1{\mu_\Delta(C)}\int_C u(t)\,d\mu_\Delta(t).
\]

The Gagliardo seminorm decomposes into intra-component and inter-component
contributions:
\begin{equation}\label{eq:component-splitting}
\begin{aligned}
[u]^p_{W^{\alpha,p}_\Delta(\T)}
&= \sum_{C\in\mathcal C}
\iint_{(C\times C)\cap\Omega_\T}
\frac{|u(t)-u(s)|^p}{|t-s|^{1+\alpha p}}
\,d(\mu_\Delta\otimes\mu_\Delta)(t,s)\\
&\quad+
\sum_{\substack{C,C'\in\mathcal C\\C\neq C'}}
\iint_{C\times C'}
\frac{|u(t)-u(s)|^p}{|t-s|^{1+\alpha p}}
\,d(\mu_\Delta\otimes\mu_\Delta)(t,s).
\end{aligned}
\end{equation}

\begin{lemma}\label{lem:cross-component-bounds}
Under \cref{ass:hybrid-class}, for every two distinct components
$C,C'\in\mathcal C$ and every $(t,s)\in C\times C'$,
\[
\frac{1}{(\operatorname{diam}\T)^{1+\alpha p}}
\le
\frac{1}{|t-s|^{1+\alpha p}}
\le
\frac{1}{\delta_0^{1+\alpha p}}.
\]
In particular, there exist constants $c_0,C_0>0$, depending only on $\T$,
$\alpha$, and $p$, such that
\begin{equation}\label{eq:cross-bounds}
c_0\iint_{C\times C'} |u(t)-u(s)|^p\,d(\mu_\Delta\otimes\mu_\Delta)
\le
\iint_{C\times C'}
\frac{|u(t)-u(s)|^p}{|t-s|^{1+\alpha p}}
\,d(\mu_\Delta\otimes\mu_\Delta)
\le
C_0\iint_{C\times C'} |u(t)-u(s)|^p\,d(\mu_\Delta\otimes\mu_\Delta).
\end{equation}
\end{lemma}

\begin{proof}
Since $C$ and $C'$ are distinct components, $|t-s|\ge\delta_0$ for all
$(t,s)\in C\times C'$, and $|t-s|\le\operatorname{diam}\T$ since $\T$ is
bounded. The bounds \eqref{eq:cross-bounds} follow by multiplying by
$|u(t)-u(s)|^p$ and integrating.
\end{proof}

\begin{lemma}[Weighted discrete Poincar\'e inequality on the component graph]\label{lem:discrete-poincare}
Let $\mathcal C$ be the finite family of connected components from \cref{ass:hybrid-class}, and set $\lambda_C:=\mu_\Delta(C)$ for $C\in\mathcal C$. Then there exists $C_{\mathrm d}>0$, depending only on the weights $(\lambda_C)_{C\in\mathcal C}$, such that for every family $(x_C)_{C\in\mathcal C}\subset\mathbb R$,
\begin{equation}\label{eq:discrete-poincare}
\sum_{C\in\mathcal C}\lambda_C\left|x_C-\frac{\sum_{C'\in\mathcal C}
\lambda_{C'}x_{C'}}{\sum_{C'\in\mathcal C}\lambda_{C'}}\right|^p
\le
C_{\mathrm d}\sum_{\substack{C,C'\in\mathcal C\\ C\neq C'}}
\lambda_C\lambda_{C'}|x_C-x_{C'}|^p.
\end{equation}
\end{lemma}

\begin{proof}
Consider on $\mathbb R^{\mathcal C}$ the two seminorms
\[
N_1(x):=\left(\sum_{C\in\mathcal C}\lambda_C\left|x_C-\frac{\sum_{C'\in\mathcal C}\lambda_{C'}x_{C'}}{\sum_{C'\in\mathcal C}\lambda_{C'}}\right|^p\right)^{1/p}
\]
and
\[
N_2(x):=\left(\sum_{\substack{C,C'\in\mathcal C\\ C\neq C'}}\lambda_C\lambda_{C'}|x_C-x_{C'}|^p\right)^{1/p}.
\]
Both vanish exactly on constant vectors. They therefore induce norms on the finite-dimensional quotient $\mathbb R^{\mathcal C}/\mathbb R$, and all norms on this quotient are equivalent. This yields \eqref{eq:discrete-poincare}.
\end{proof}

\begin{theorem}[Poincar\'e-type inequality]\label{thm:poincare}
Under \cref{ass:hybrid-class}, there exists a constant $C_P>0$, depending
only on $\T$, $\alpha$, and $p$, such that
\begin{equation}\label{eq:poincare-main}
\|u-u_{\T}\|_{L^p_{\Delta}(\T)}
\le
C_P\,[u]_{W^{\alpha,p}_{\Delta}(\T)}
\qquad
\text{for every }u\in W^{\alpha,p}_{\Delta}(\T).
\end{equation}
\end{theorem}

\begin{proof}
For each $C\in\mathcal C$, we write
$u-u_{\T}=(u-u_C)+(u_C-u_{\T})$. Using
$|x+y|^p\le 2^{p-1}(|x|^p+|y|^p)$ and integrating over
$\T=\bigcup_{C\in\mathcal C} C$, we obtain
\begin{equation}\label{eq:poincare-decomposition}
\|u-u_{\T}\|_{L_\Delta^p(\T)}^p
\le
2^{p-1}\sum_{C\in\mathcal C}\int_C |u-u_C|^p\,d\mu_\Delta
+
2^{p-1}\sum_{C\in\mathcal C}\mu_\Delta(C)\,|u_C-u_{\T}|^p.
\end{equation}

\medskip
\noindent
\textbf{Step 1: oscillation inside each component.}
Let $C\in\mathcal C$. If $C=I_\ell=[a_\ell,b_\ell]$ is an interval
component, then $\mu_\Delta$ restricted to $I_\ell$ coincides with the
Lebesgue measure on $[a_\ell,b_\ell]$; see
\cite[Theorem~5.2]{CabadaVivero2006}. Therefore
$u|_{I_\ell}\in W^{\alpha,p}(I_\ell)$ in the classical one-dimensional
Gagliardo sense, and the standard fractional Poincar\'e inequality on
bounded intervals (see, e.g.,
\cite[Theorem~6.7]{DiNezzaPalatucciValdinoci2012}) yields
\[
\int_{I_\ell}|u-u_{I_\ell}|^p\,dt
\le
C_\ell
\iint_{I_\ell\times I_\ell}
\frac{|u(t)-u(s)|^p}{|t-s|^{1+\alpha p}}\,dt\,ds,
\]
where $C_\ell$ depends only on $|I_\ell|$, $\alpha$, and $p$. Since
$\mu_\Delta|_{I_\ell}$ equals the Lebesgue measure, both sides can
equivalently be written in terms of $d\mu_\Delta$. If $C=\{d_j\}$ is a
singleton, then $u-u_C=0$ $\mu_\Delta$-a.e.\ on $C$. Since $\mathcal C$
is finite, summing over all $C\in\mathcal C$ with a single constant $C>0$
bounding all $C_\ell$ gives
\[
\sum_{C\in\mathcal C}\int_C |u-u_C|^p\,d\mu_\Delta
\le
C\sum_{C\in\mathcal C}
\iint_{(C\times C)\cap\Omega_{\T}}
\frac{|u(t)-u(s)|^p}{|t-s|^{1+\alpha p}}
\,d(\mu_\Delta\otimes\mu_\Delta)
\le
C\,[u]_{W_\Delta^{\alpha,p}(\T)}^p.
\]

\medskip
\noindent
\textbf{Step 2: control of the component averages.}
Set $\lambda_C:=\mu_\Delta(C)$ for $C\in\mathcal C$. Applying \cref{lem:discrete-poincare} with $x_C=u_C$ and using
$u_{\T}=(\sum_C \lambda_C u_C)/(\sum_C \lambda_C)$, we get
\[
\sum_{C\in\mathcal C}\mu_\Delta(C)\,|u_C-u_{\T}|^p
\le
C_{\mathrm d}\sum_{\substack{C,C'\in\mathcal C\\ C\neq C'}}
\mu_\Delta(C)\mu_\Delta(C')\,|u_C-u_{C'}|^p.
\]
For $C\neq C'$, Jensen's inequality applied to the probability measure
$(\mu_\Delta(C)\mu_\Delta(C'))^{-1}\,d(\mu_\Delta\otimes\mu_\Delta)$ on
$C\times C'$ gives
\[
|u_C-u_{C'}|^p
\le
\frac1{\mu_\Delta(C)\mu_\Delta(C')}
\iint_{C\times C'}|u(t)-u(s)|^p\,d(\mu_\Delta\otimes\mu_\Delta)(t,s).
\]
Multiplying by $\mu_\Delta(C)\mu_\Delta(C')$, summing over $C\neq C'$,
and applying \cref{lem:cross-component-bounds}, we obtain
\[
\sum_{C\in\mathcal C}\mu_\Delta(C)\,|u_C-u_{\T}|^p
\le
C\,[u]_{W_\Delta^{\alpha,p}(\T)}^p.
\]

\medskip
\noindent
\textbf{Step 3: conclusion.}
Combining Steps~1 and~2 with \eqref{eq:poincare-decomposition} and taking
the $p$th root yields \eqref{eq:poincare-main}.
\end{proof}

\begin{corollary}\label{cor:coercive-mod-constants}
Under \cref{ass:hybrid-class}, there exists $C>0$ such that
\[
\|u\|_{L^p_{\Delta}(\T)}
\le
C\Big(
[u]_{W^{\alpha,p}_{\Delta}(\T)}
+
|u_{\T}|
\Big)
\qquad
\text{for every }u\in W^{\alpha,p}_{\Delta}(\T).
\]
\end{corollary}

\begin{proof}
Write $u=(u-u_{\T})+u_{\T}$, apply the triangle inequality and
\cref{thm:poincare}, and bound
$\|u_{\T}\|_{L^p_\Delta(\T)}=|u_{\T}|\,\mu_\Delta(\T)^{1/p}$.
\end{proof}

The Poincar\'e inequality of \cref{thm:poincare} controls the $L^p$-norm
of $u-u_\T$ by the Gagliardo seminorm, with a constant $C_P$ depending
on the geometry of $\T$. Proposition~\ref{prop:poincare-gap} below makes
this geometric dependence quantitative for a two-component family
$\T_\delta=[0,1]\cup[1+\delta,2+\delta]$, showing that $C_P(\delta)$
grows at least polynomially as the gap width $\delta$ increases.
A natural further question is whether the Gagliardo seminorm also
provides a \emph{gain of integrability}, that is, an embedding into
$L^q_\Delta(\T)$ for exponents $q>p$. The following
\S\ref{subsec:sobolev} shows that this is indeed the case when
$\alpha p<1$, with the critical exponent
$p^*_\alpha=p/(1-\alpha p)$ inherited from the classical
one-dimensional theory on each interval component.

\begin{proposition}[Dependence of the Poincar\'e constant on the gap]
\label{prop:poincare-gap}
Let $\alpha\in(0,1)$, $p=2$, and
$\T_\delta=[0,1]\cup[1+\delta,2+\delta]$ for $\delta>0$. Define
\[
K(\delta,\alpha) :=
\begin{cases}
\dfrac{2(1+\delta)^{1-2\alpha}-\delta^{1-2\alpha}-(2+\delta)^{1-2\alpha}}
{2\alpha(1-2\alpha)}
& \text{if }\alpha\neq\tfrac{1}{2},\\[8pt]
\ln\!\left(\dfrac{(1+\delta)^2}{\delta(2+\delta)}\right)
& \text{if }\alpha=\tfrac{1}{2}.
\end{cases}
\]
Then the optimal Poincar\'e constant in
$\|u-u_{\T_\delta}\|_{L^2_\Delta}\le C_P(\delta)[u]_{W^{\alpha,2}_\Delta(\T_\delta)}$
satisfies
\begin{equation}\label{eq:poincare-lower}
C_P(\delta) \ge \frac{1}{2\,K(\delta,\alpha)^{1/2}}.
\end{equation}
Moreover, as $\delta\to\infty$,
\begin{equation}\label{eq:poincare-blowup}
C_P(\delta) \ge \frac{1}{2}\,\delta^{(1+2\alpha)/2}\bigl(1+o(1)\bigr).
\end{equation}
\end{proposition}

\begin{proof}
Consider the piecewise-constant test function
$u_\delta = \mathbf{1}_{[0,1]} - \mathbf{1}_{[1+\delta,2+\delta]}$.
Then $(u_\delta)_{\T_\delta}=0$,
$\|u_\delta\|^2_{L^2_\Delta(\T_\delta)}=2$, and
\[
[u_\delta]^2_{W^{\alpha,2}_\Delta(\T_\delta)}
= 2\iint_{[0,1]\times[1+\delta,2+\delta]}
\frac{4}{|t-s|^{1+2\alpha}}\,dt\,ds
= 8\,K(\delta,\alpha),
\]
since the intra-component contributions vanish (each
restriction is constant). The Poincar\'e inequality gives
$2 \le C_P(\delta)^2 \cdot 8\,K(\delta,\alpha)$, hence
\eqref{eq:poincare-lower}.

Consider the piecewise-constant test function
$u_\delta = \mathbf{1}_{[0,1]} - \mathbf{1}_{[1+\delta,2+\delta]}$.
Then $(u_\delta)_{\T_\delta}=0$,
$\|u_\delta\|^2_{L^2_\Delta(\T_\delta)}=2$, and
\[
[u_\delta]^2_{W^{\alpha,2}_\Delta(\T_\delta)}
= 2\iint_{[0,1]\times[1+\delta,2+\delta]}
\frac{4}{|t-s|^{1+2\alpha}}\,dt\,ds
= 8\,K(\delta,\alpha),
\]
since the intra-component contributions vanish (each
restriction is constant). The Poincar\'e inequality gives
$2 \le C_P(\delta)^2 \cdot 8\,K(\delta,\alpha)$, hence
\eqref{eq:poincare-lower}.

For the asymptotic \eqref{eq:poincare-blowup}, we expand each term of
$f(\delta):=2(1+\delta)^{1-2\alpha}-\delta^{1-2\alpha}-(2+\delta)^{1-2\alpha}$
as $\delta\to\infty$ using the binomial expansion
$(1+x)^\beta = 1 + \beta x + \tfrac{\beta(\beta-1)}{2}x^2 + O(x^3)$
with $\beta=1-2\alpha$, $x=\delta^{-1}$ and $x=2\delta^{-1}$ respectively:
\begin{align*}
2(1+\delta)^{1-2\alpha}
&= 2\delta^{1-2\alpha}
  \!\left[1 + \frac{1-2\alpha}{\delta}
   + \frac{(1-2\alpha)(-2\alpha)}{2\,\delta^2}
   + O(\delta^{-3})\right],\\
\delta^{1-2\alpha}
&= \delta^{1-2\alpha},\\
(2+\delta)^{1-2\alpha}
&= \delta^{1-2\alpha}
  \!\left[1 + \frac{2(1-2\alpha)}{\delta}
   + \frac{4(1-2\alpha)(-2\alpha)}{2\,\delta^2}
   + O(\delta^{-3})\right].
\end{align*}
Subtracting:
\[
f(\delta)
= \delta^{1-2\alpha}
\left[
  \underbrace{2-1-1}_{=\,0}
  + \frac{1}{\delta}\underbrace{2(1-2\alpha)-2(1-2\alpha)}_{=\,0}
  + \frac{(1-2\alpha)(-2\alpha)}{2\,\delta^2}
    \underbrace{(2 - 4)}_{=-2}
  + O(\delta^{-3})
\right],
\]
so
\[
f(\delta)
= \frac{(1-2\alpha)(-2\alpha)(-2)}{2}\cdot\delta^{1-2\alpha-2}
  + O(\delta^{-(2+2\alpha)})
= \frac{2\alpha(1-2\alpha)}{\delta^{1+2\alpha}}
  + O(\delta^{-(2+2\alpha)}).
\]
Hence $K(\delta,\alpha)=f(\delta)/(2\alpha(1-2\alpha))\sim\delta^{-(1+2\alpha)}$,
which substituted into \eqref{eq:poincare-lower} yields \eqref{eq:poincare-blowup}.
\end{proof}
\begin{proposition}[Two-sided gap estimate for the Poincar\'e constant]
\label{prop:two-sided-poincare}
In the setting of Proposition~\ref{prop:poincare-gap} ($p = 2$, 
$\mathbb{T}_\delta = [0,1] \cup [1+\delta, 2+\delta]$), there exists a 
constant $C^* = C^*(\alpha) > 0$ such that, for every $\delta \ge 1$,
\begin{equation}\label{eq:two-sided-poincare}
\frac{1}{2 \, K(\delta, \alpha)^{1/2}} 
\;\le\; C_P(\delta) 
\;\le\; C^* \, \delta^{(1+2\alpha)/2}.
\end{equation}
In particular, $C_P(\delta) \asymp \delta^{(1+2\alpha)/2}$ as 
$\delta \to \infty$, and the lower bound \eqref{eq:poincare-lower} 
captures the sharp asymptotic rate.
\end{proposition}

\begin{proof}
The lower bound was established in Proposition~\ref{prop:poincare-gap}. 
We prove the upper bound.

Let $u \in W^{\alpha,2}_\Delta(\mathbb{T}_\delta)$ with 
$u_{\mathbb{T}_\delta} = 0$. Write $\mathbb{T}_\delta = I_1 \cup I_2$ with 
$I_1 = [0,1]$ and $I_2 = [1+\delta, 2+\delta]$, and let 
$u_i := u_{I_i}$ denote the average on $I_i$. Since 
$\mu_\Delta(I_1) = \mu_\Delta(I_2) = 1$ and $\mu_\Delta(\mathbb{T}_\delta) = 2$, 
the global vanishing average condition $u_{\mathbb{T}_\delta} = 0$ gives
\[
u_1 + u_2 = 0.
\]

\smallskip
\noindent\textbf{Step 1: $L^2$ decomposition.} Writing 
$u = (u - u_i) + u_i$ on each $I_i$,
\[
\|u\|^2_{L^2_\Delta(\mathbb{T}_\delta)} 
= \sum_{i=1}^2 \|u - u_i\|^2_{L^2(I_i)} + \sum_{i=1}^2 |u_i|^2 \, \mu_\Delta(I_i)
= \sum_{i=1}^2 \|u - u_i\|^2_{L^2(I_i)} + 2|u_1|^2,
\]
where the last equality uses $|u_1|^2 = |u_2|^2$.

\smallskip
\noindent\textbf{Step 2: control of intra-component oscillation.} On each 
interval $I_i$ of length~$1$ (independent of $\delta$), the classical 
fractional Poincar\'e inequality on bounded intervals gives
\[
\|u - u_i\|^2_{L^2(I_i)} 
\le C_1(\alpha) \, [u]^2_{W^{\alpha,2}(I_i)} 
\le C_1(\alpha) \, [u]^2_{W^{\alpha,2}_\Delta(\mathbb{T}_\delta)},
\]
where $C_1(\alpha)$ is independent of $\delta$.

\smallskip
\noindent\textbf{Step 3: control of the average.} Using $u_2 = -u_1$ and 
Jensen's inequality on $I_1 \times I_2$,
\[
|u_1|^2 = \frac{|u_1 - u_2|^2}{4} 
\le \frac{1}{4} \iint_{I_1 \times I_2} |u(t) - u(s)|^2 \, dt \, ds.
\]
For $(t,s) \in I_1 \times I_2$ and $\delta \ge 1$, one has 
$\delta \le |t-s| \le 2 + \delta \le 3\delta$, so
\[
|t-s|^{1+2\alpha} \le (3\delta)^{1+2\alpha} = 3^{1+2\alpha} \delta^{1+2\alpha},
\]
and therefore
\[
\iint_{I_1 \times I_2} |u(t) - u(s)|^2 \, dt \, ds 
\le 3^{1+2\alpha} \delta^{1+2\alpha} 
\iint_{I_1 \times I_2} \frac{|u(t) - u(s)|^2}{|t-s|^{1+2\alpha}} \, dt \, ds 
\le 3^{1+2\alpha} \delta^{1+2\alpha} \, [u]^2_{W^{\alpha,2}_\Delta(\mathbb{T}_\delta)}.
\]
Hence
\[
2|u_1|^2 \le \frac{3^{1+2\alpha}}{2} \delta^{1+2\alpha} \, 
[u]^2_{W^{\alpha,2}_\Delta(\mathbb{T}_\delta)}.
\]

\smallskip
\noindent\textbf{Step 4: combination.} Summing the bounds of Steps~2 and~3,
\[
\|u\|^2_{L^2_\Delta(\mathbb{T}_\delta)} 
\le \Big( 2 C_1(\alpha) + \frac{3^{1+2\alpha}}{2} \delta^{1+2\alpha} \Big) 
\, [u]^2_{W^{\alpha,2}_\Delta(\mathbb{T}_\delta)}.
\]
For $\delta \ge 1$, the right-hand side is bounded by 
$C^*(\alpha) \, \delta^{1+2\alpha} \, [u]^2_{W^{\alpha,2}_\Delta(\mathbb{T}_\delta)}$ 
with $C^*(\alpha) := 2 C_1(\alpha) + \tfrac{1}{2} 3^{1+2\alpha}$. Taking square 
roots and the supremum over admissible $u$ yields
\[
C_P(\delta) \le \big(C^*(\alpha)\big)^{1/2} \, \delta^{(1+2\alpha)/2},
\]
which proves the upper bound in \eqref{eq:two-sided-poincare}.
\end{proof}

\begin{remark}\label{rem:bbm-sharp}
The two-sided estimate of Proposition~\ref{prop:two-sided-poincare} shows 
that the test function 
$u_\delta = \mathbf{1}_{[0,1]} - \mathbf{1}_{[1+\delta, 2+\delta]}$ used in 
the proof of Proposition~\ref{prop:poincare-gap} is asymptotically optimal: 
it captures the dominant mode in the spectral structure of the Poincar\'e 
quotient as $\delta \to \infty$. The geometric phenomenon is now fully 
quantified: the Poincar\'e constant grows precisely at rate 
$\delta^{(1+2\alpha)/2}$ with the gap, with matching constants up to a 
universal factor depending only on $\alpha$.
\end{remark}
\begin{remark}\label{rem:gap-general-p}\leavevmode
\begin{enumerate}
\item
For general \(1\le p<\infty\), the same test function as in
\cref{prop:poincare-gap},
\[
u_\delta=\mathbf{1}_{[0,1]}-\mathbf{1}_{[1+\delta,2+\delta]},
\]
gives
\[
[u_\delta]^p_{W^{\alpha,p}_\Delta(\T_\delta)}
= 2^p\iint_{[0,1]\times[1+\delta,2+\delta]}
|t-s|^{-(1+\alpha p)}\,dt\,ds
= 2^p\,K_p(\delta,\alpha),
\]
where
\[
K_p(\delta,\alpha)\sim\delta^{-(1+\alpha p)}
\qquad\text{as }\delta\to\infty,
\]
by the same Taylor expansion as in the proof of
\cref{prop:poincare-gap}. Hence
\[
C_P(\delta)\ge c\,\delta^{(1+\alpha p)/p},
\]
that is,
\[
C_P(\delta)\ge c\,\delta^{\alpha+1/p}.
\]
In particular, the blow-up exponent increases with \(\alpha\), while for fixed
\(\alpha\) it decreases with \(p\).

\item
\cref{prop:poincare-gap} reveals a quantitative mechanism that is specific to
the time-scale setting and has no counterpart in the classical theory on
\(\mathbb{R}\): the Poincar\'e constant grows like
\[
\delta^{(1+2\alpha)/2}
\]
with the gap separating two components, and the exponent increases with the
fractional order \(\alpha\). In other words, \emph{higher-order Gagliardo
spaces are more sensitive to the geometric fragmentation of the time scale}.
This is consistent with the fact that, when \(\alpha\) increases, the kernel
\[
|t-s|^{-(1+2\alpha)}
\]
decays faster at large distances. As a consequence, the interaction between
well-separated components becomes weaker more rapidly, so that the Poincar\'e
constant grows more strongly with the gap.
\end{enumerate}
\end{remark}
\begin{example}\label{ex:poincare-table}

Table~\ref{tab:poincare} displays the lower bound \eqref{eq:poincare-lower} for several values of
\(\delta\) and \(\alpha\), illustrating the polynomial blow-up of the
Poincar\'e constant as the gap width increases.

The numerical values in Table~\ref{tab:poincare} should be read as more than a simple list of examples. First, for each fixed \(\alpha\), the lower bound increases steadily with \(\delta\), which is exactly the behavior predicted by the geometric argument behind \eqref{eq:poincare-lower}. As the two connected components move farther apart, the cross-interaction measured by the Gagliardo seminorm becomes weaker, so a function taking opposite signs on the two intervals can keep a comparable \(L^2\)-mass while paying less nonlocal energy. This progressive loss of interaction is precisely what drives the growth of the Poincar\'e constant.

Second, for large values of \(\delta\), the dependence on \(\alpha\) is consistent with the theoretical scaling. Larger values of \(\alpha\) produce a stronger asymptotic blow-up. This reflects the fact that the kernel
\[
|t-s|^{-1-2\alpha}
\]
decays faster at large distances when \(\alpha\) increases, so that the coupling between well-separated components is damped more severely. In this sense, higher-order fractional seminorms are more sensitive to the fragmentation of the time scale.

Most importantly, Table~\ref{tab:poincare} confirms the power-law mechanism captured by \cref{prop:poincare-gap}. The growth observed numerically is consistent with the exponent \((1+2\alpha)/2\) in the case \(p=2\), and with the more general exponent \(\alpha+1/p\) described in \cref{rem:gap-general-p}. Therefore, the table supports the interpretation that the blow-up of the Poincar\'e constant is a geometric effect of the gap between the connected components.
\begin{table}[ht]
\centering
\renewcommand{\arraystretch}{1.15}
\begin{tabular}{r|ccccccc}
\hline
$\delta$
& $\alpha\!=\!0.10$
& $\alpha\!=\!0.25$
& $\alpha\!=\!0.40$
& $\alpha\!=\!0.50$
& $\alpha\!=\!0.60$
& $\alpha\!=\!0.75$
& $\alpha\!=\!0.90$ \\
\hline
$0.1$   & $0.459$ & $0.434$ & $0.402$ & $0.378$
        & $0.352$ & $0.310$ & $0.268$ \\
$1$     & $0.735$ & $0.805$ & $0.880$ & $0.932$
        & $0.987$ & $1.07$ & $1.16$ \\
$5$     & $1.46$ & $1.91$ & $2.49$ & $2.98$
        & $3.56$ & $4.65$ & $6.07$ \\
$10$    & $2.11$ & $3.02$ & $4.32$ & $5.49$
        & $6.97$ & $9.99$ & $14.3$ \\
$50$    & $5.29$ & $9.54$ & $17.2$ & $25.5$
        & $37.8$ & $68.1$ & $123$  \\
$100$   & $7.97$ & $15.9$ & $31.8$ & $50.5$
        & $80.1$ & $160$ & $320$  \\
\hline
\\[-6pt]
\multicolumn{8}{c}{\footnotesize
Predicted rate:
$C_P(\delta)\sim\tfrac12\,\delta^{(1+2\alpha)/2}$} \\
\hline
\end{tabular}
\caption{Lower bound on $C_P(\delta)$ for
$\T_\delta=[0,1]\cup[1+\delta,2+\delta]$, $p=2$.
The growth rate accelerates with $\alpha$, confirming
that higher fractional regularity amplifies the
sensitivity to geometric separation.}
\label{tab:poincare}
\end{table}
\end{example}

\subsection{A brief numerical illustration}
\label{subsec:numerical}

The following computations are included only as a short illustration of the geometric content of the Poincar\'e inequality. They are not intended as a separate numerical study, and no rigorous discretization error analysis is claimed.

We consider the hybrid family
\[
\T_\delta=[0,1]\cup\{1+\delta\}\cup[2+\delta,3+\delta],
\qquad \delta>0,
\]
and, for a fixed order $0<\alpha<\tfrac12$, the first nonlocal Poincar\'e eigenvalue
\[
\lambda_1(\delta)
:=
\inf_{\substack{u\in W^{\alpha,2}_{\Delta}(\T_\delta)\\ \int_{\T_\delta}u\,d\mu_\Delta=0,\ u\neq0}}
\frac{
\displaystyle
\iint_{\Omega_{\T_\delta}}
\frac{|u(t)-u(s)|^2}{|t-s|^{1+2\alpha}}
\,d(\mu_\Delta\otimes\mu_\Delta)(t,s)
}{
\displaystyle
\int_{\T_\delta}|u(t)|^2\,d\mu_\Delta(t)
}.
\]
The corresponding optimal Poincar\'e constant is $C_P(\delta)=\lambda_1(\delta)^{-1/2}$.

We approximate $\lambda_1(\delta)$ by a cell-based discretization: each interval component is partitioned into $N$ uniform cells of size $h=1/N$, while the isolated point $\{1+\delta\}$ is retained as an exact discrete node. The associated mass matrix uses the exact $\Delta$-mass of the isolated point, namely $\mu_\Delta(\{1+\delta\})=1$, and the stiffness matrix is assembled from the exact interaction coefficients between cells and between cells and the isolated node.

For this hybrid eigenvalue computation, the restriction $0<\alpha<1/2$ is dictated by the use of a piecewise constant trial space on the interval components: in that regime, jumps still have finite nonlocal energy, whereas a conforming treatment of $\alpha\ge 1/2$ would require a continuous piecewise affine discretization. In the computations below we take $\alpha=0.35$ and $N=150$.

\begin{figure}[tbp]
\centering

\begin{minipage}{0.48\textwidth}
\centering
\includegraphics[
width=\textwidth,
height=4.2cm,
keepaspectratio
]{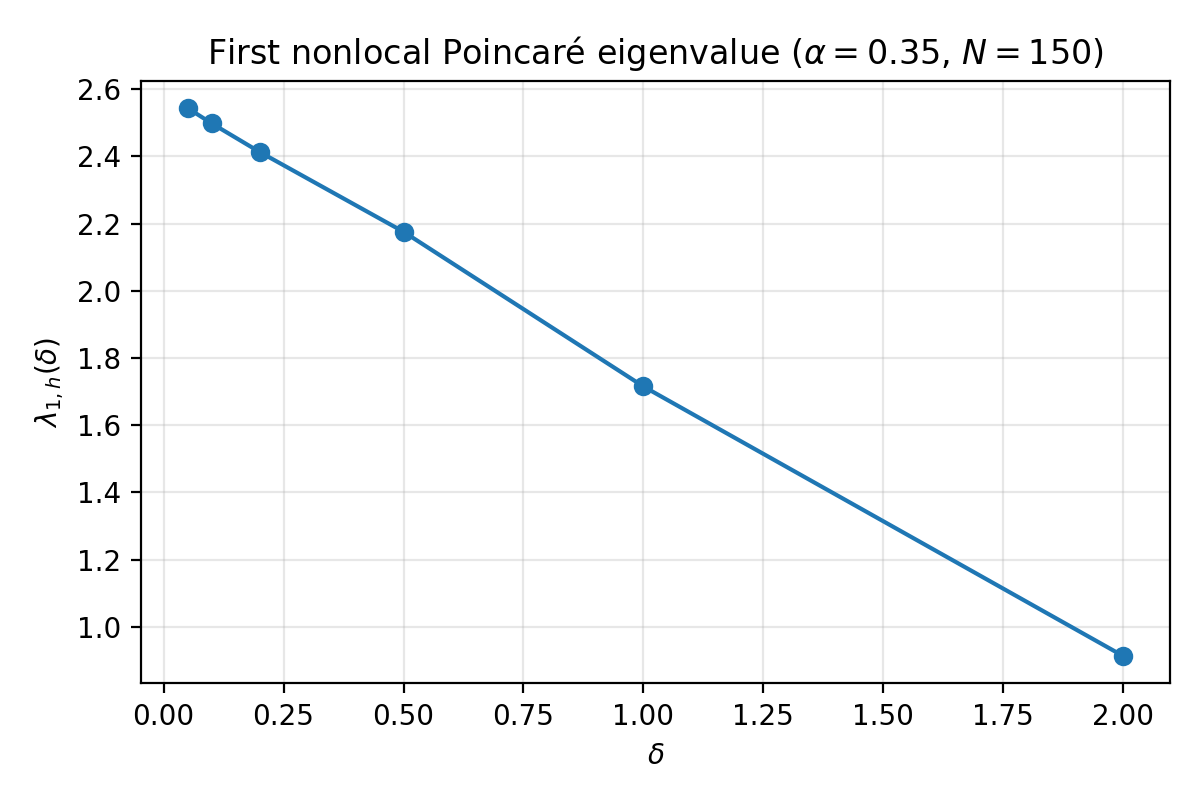}

\vspace{0.15cm}

\small
(a) First nonlocal Poincar\'e eigenvalue
$\lambda_{1,h}(\delta)$.
\end{minipage}
\hfill
\begin{minipage}{0.48\textwidth}
\centering

\includegraphics[
width=\textwidth,
height=4.2cm,
keepaspectratio
]{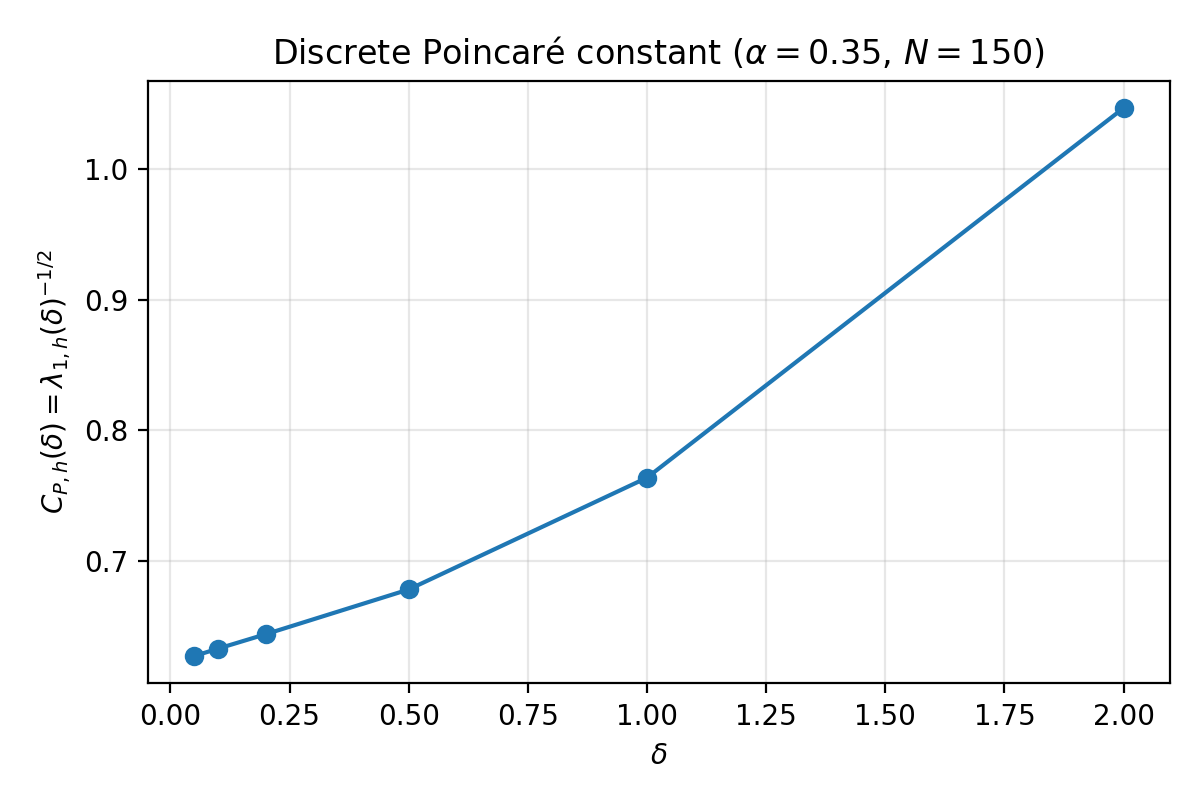}

\vspace{0.15cm}

\small
(b) Discrete Poincar\'e constant
$C_{P,h}(\delta)=\lambda_{1,h}(\delta)^{-1/2}$.
\end{minipage}

\caption{Geometric dependence of the first nonlocal Poincar\'e eigenvalue
and of the associated discrete Poincar\'e constant for the family
$\T_\delta=[0,1]\cup\{1+\delta\}\cup[2+\delta,3+\delta]$,
with $\alpha=0.35$ and $N=150$.}

\label{fig:lambda-cp}

\end{figure}

Figure~\ref{fig:lambda-cp} is consistent with the geometric intuition behind \cref{thm:poincare,prop:poincare-gap}: as the connected components become more separated, the nonlocal coupling weakens, the first eigenvalue decreases, and the corresponding Poincar\'e constant increases. We regard this subsection only as an illustration of the mechanism already captured by the theoretical estimates.

\subsection{Sobolev-type embedding on bounded hybrid time scales}
\label{subsec:sobolev}

We begin with an auxiliary estimate that lifts the cross-component control
from $L^p$ to $L^q$ for $q\ge p$.

\begin{lemma}\label{lem:cross-Lq}
Under \cref{ass:hybrid-class}, for every $q\ge p$ and every two distinct
components $C,C'\in\mathcal C$,
\[
|u_C - u_{C'}|^q
\le
\left(\frac{1}{\mu_\Delta(C)\,\mu_\Delta(C')}
\iint_{C\times C'}|u(t)-u(s)|^p\,
d(\mu_\Delta\otimes\mu_\Delta)\right)^{q/p}.
\]
\end{lemma}

\begin{proof}
By Jensen's inequality,
$|u_C-u_{C'}|^p \le
(\mu_\Delta(C)\mu_\Delta(C'))^{-1}
\iint_{C\times C'}|u(t)-u(s)|^p\,d(\mu_\Delta\otimes\mu_\Delta)$.
Since $q\ge p$, the map $x\mapsto x^{q/p}$ is convex and nondecreasing on
$[0,\infty)$; raising both sides to the power $q/p$ gives the claim.
\end{proof}

\begin{theorem}[Fractional Sobolev embedding]\label{thm:embedding}
Under \cref{ass:hybrid-class}, let $\alpha\in(0,1)$ and $1\le p<\infty$
with $\alpha p < 1$. Set
\[
p^*_\alpha := \frac{p}{1-\alpha p}.
\]
Then for every $1\le q \le p^*_\alpha$, there exists $C>0$, depending only
on $\T$, $\alpha$, $p$, and $q$, such that
\begin{equation}\label{eq:sobolev-embedding}
\|u - u_\T\|_{L^q_\Delta(\T)}
\le C\,[u]_{W^{\alpha,p}_\Delta(\T)}
\end{equation}
for every $u\in W^{\alpha,p}_\Delta(\T)$. In particular,
$W^{\alpha,p}_\Delta(\T)$ embeds continuously into $L^q_\Delta(\T)$ for
every $q\in[1,\,p^*_\alpha]$.
\end{theorem}

\begin{proof}
If $1\le q<p$, then the finiteness of $\mu_\Delta(\T)$ and \cref{thm:poincare} give
\[
\|u-u_\T\|_{L^q_\Delta(\T)}
\le \mu_\Delta(\T)^{1/q-1/p}\|u-u_\T\|_{L^p_\Delta(\T)}
\le C [u]_{W^{\alpha,p}_\Delta(\T)}.
\]
Thus it remains to prove the estimate for $p\le q\le p_\alpha^*$.

We treat the two types of connected components separately and then combine.

\medskip\noindent
\textbf{Step 1: interval components.}
Let $I_\ell=[a_\ell,b_\ell]$ be an interval component. Since
$\mu_\Delta|_{I_\ell}$ coincides with the Lebesgue measure,
$u|_{I_\ell}\in W^{\alpha,p}(I_\ell)$ in the classical Gagliardo sense.
Each $I_\ell$ is a bounded interval (by Assumption~\ref{ass:hybrid-class}),
so both the sub-critical and endpoint one-dimensional fractional Sobolev
inequalities apply directly:
by \cite[Theorem~6.7]{DiNezzaPalatucciValdinoci2012} for $q<p^*_\alpha$
and by \cite[Theorem~8.2]{DiNezzaPalatucciValdinoci2012} for $q=p^*_\alpha$
(the boundedness of $I_\ell$ being an explicit hypothesis of that theorem),
\[
\|u - u_{I_\ell}\|_{L^q(I_\ell)}
\le C_\ell\,[u]_{W^{\alpha,p}(I_\ell)}
\le C_\ell\,[u]_{W^{\alpha,p}_\Delta(\T)}
\]
for every $q\le p^*_\alpha$, where $C_\ell$ depends on $|I_\ell|$,
$\alpha$, $p$, and $q$.

\medskip\noindent
\textbf{Step 2: singleton components.}
If $C=\{d_j\}$, then $\|u - u_C\|^q_{L^q_\Delta(C)} = 0$ for any
$q\ge 1$.

\medskip\noindent
\textbf{Step~3: global estimate.}
We write $u-u_\T=(u-u_C)+(u_C-u_\T)$ on each $C\in\mathcal C$.
By the triangle inequality in $L^q_\Delta(\T)$,
\[
\|u-u_\T\|_{L^q_\Delta(\T)}
\le
\left(\sum_{C\in\mathcal C}\|u-u_C\|^q_{L^q_\Delta(C)}\right)^{1/q}
+\left(\sum_{C\in\mathcal C}\mu_\Delta(C)|u_C-u_\T|^q\right)^{1/q}.
\]
The first term is bounded by $C[u]_{W^{\alpha,p}_\Delta(\T)}$
by Steps~1--2.

For the second term, the discrete weighted Poincar\'e inequality
\eqref{eq:discrete-poincare} with exponent $q$ in place of $p$ gives
\[
\sum_{C\in\mathcal C}\mu_\Delta(C)|u_C-u_\T|^q
\le C_1\sum_{\substack{C,C'\in\mathcal C\\C\neq C'}}
\mu_\Delta(C)\mu_\Delta(C')|u_C-u_{C'}|^q.
\]
By Lemma~\ref{lem:cross-Lq},
\[
|u_C-u_{C'}|^q
\le
\left(\frac{1}{\mu_\Delta(C)\mu_\Delta(C')}
\iint_{C\times C'}|u(t)-u(s)|^p\,d(\mu_\Delta\otimes\mu_\Delta)
\right)^{q/p}.
\]
Let
\[
m_*:=\min_{C\in\mathcal C}\mu_\Delta(C)>0.
\]
Since $q\ge p$, one has $1-q/p\le 0$, and therefore
\[
(\mu_\Delta(C)\mu_\Delta(C'))^{1-q/p}\le m_*^{2(1-q/p)}.
\]
Multiplying by $\mu_\Delta(C)\mu_\Delta(C')$ yields
\[
\mu_\Delta(C)\mu_\Delta(C')|u_C-u_{C'}|^q
\le
m_*^{2(1-q/p)}
\left(\iint_{C\times C'}|u(t)-u(s)|^p\,
d(\mu_\Delta\otimes\mu_\Delta)\right)^{q/p}.
\]
By Lemma~\ref{lem:cross-component-bounds} (equation \eqref{eq:cross-bounds}),
\[
\iint_{C\times C'}|u(t)-u(s)|^p\,d(\mu_\Delta\otimes\mu_\Delta)
\le C_0^{-1}
\iint_{C\times C'}\frac{|u(t)-u(s)|^p}{|t-s|^{1+\alpha p}}\,
d(\mu_\Delta\otimes\mu_\Delta)
\le C_0^{-1}[u]^p_{W^{\alpha,p}_\Delta(\T)}.
\]
Since $\mathcal C$ is finite, summing over all $C\neq C'$ gives
\[
\sum_{C\in\mathcal C}\mu_\Delta(C)|u_C-u_\T|^q
\le C\,[u]^q_{W^{\alpha,p}_\Delta(\T)}.
\]
Taking the $q$-th root and combining with the first term yields
\eqref{eq:sobolev-embedding}.
\end{proof}

\begin{remark}[Dependence of the embedding constant on the geometry]\label{rem:embedding-constant-m-star}
The constant $C$ in~\eqref{eq:sobolev-embedding} is not uniform in
the geometry of $\T$: it depends, in particular, on the quantity
$m_* := \min_{C\in\mathcal C}\mu_\Delta(C)>0$ through a factor of the
form $m_*^{2(1-q/p)}$. Since $q\ge p$ this factor grows as $m_*\to 0$,
reflecting the fact that very light components couple weakly through
the inter-component kernel and force a larger constant in the
cross-component estimate. This dependence is intrinsic to the
component-decomposition argument and cannot be removed without
additional assumptions (for instance, a uniform lower bound on
component masses, which is part of
Assumption~\ref{ass:hybrid-class}). The constant is therefore best
viewed as depending on the full geometric data
$(\T,\alpha,p,q,\delta_0,m_*,\mathrm{diam}\,\T)$, where $\delta_0$
is the inter-component separation. The endpoint $q=p^*_\alpha$,
in particular, requires the one-dimensional endpoint embedding on
each interval component~\cite[Theorem~8.2]{DiNezzaPalatucciValdinoci2012}
and the same dependence on $m_*$.
\end{remark}

\begin{remark}\label{rem:morrey-trudinger}
When $\alpha p = 1$, the classical theory gives embedding into $L^q$ for
every $q<\infty$ but not into $L^\infty$ (Trudinger-type inequality see \cite[Theorem~8.5]{DiNezzaPalatucciValdinoci2012}). When
$\alpha p > 1$, the Morrey embedding gives
$W^{\alpha,p}(I_\ell)\hookrightarrow C^{0,\alpha-1/p}(I_\ell)$ on each
interval component, which yields
$W^{\alpha,p}_\Delta(\T)\hookrightarrow L^\infty_\Delta(\T)$. 
\end{remark}

Having established both the Poincar\'e inequality and the Sobolev embedding,
we now investigate weighted estimates. The next subsection introduces a
fractional Hardy inequality with subcritical weight and derives, as a
consequence, a Caffarelli--Kohn--Nirenberg-type interpolation estimate. These
results illustrate that the Gagliardo framework on time scales supports a
nontrivial weighted theory in which the three estimates---Poincar\'e,
Sobolev embedding, and Hardy---fit together.
\subsection{A fractional Rellich--Kondrachov compact embedding}

The Sobolev embedding of Theorem~\ref{thm:embedding} can be 
strengthened to a compact embedding under the same geometric assumption. 
This is the key ingredient for variational applications.

\begin{theorem}[Fractional Rellich--Kondrachov on hybrid time scales]
\label{thm:rellich}
Let $\mathbb{T}$ satisfy Assumption~\ref{ass:hybrid-class}, let $\alpha \in (0,1)$, 
and let $1 \le p < \infty$.
\begin{enumerate}
\item[(i)] If $\alpha p < 1$, then for every $q \in [1, p^*_\alpha)$ the embedding
\[
W^{\alpha,p}_\Delta(\mathbb{T}) \hookrightarrow L^q_\Delta(\mathbb{T})
\]
is compact.
\item[(ii)] If $\alpha p \ge 1$, then for every $q \in [1, \infty)$ the embedding 
$W^{\alpha,p}_\Delta(\mathbb{T}) \hookrightarrow L^q_\Delta(\mathbb{T})$ is compact.
\end{enumerate}
\end{theorem}

\begin{proof}
Let $(u_n) \subset W^{\alpha,p}_\Delta(\mathbb{T})$ be a bounded sequence,
\[
\sup_n \|u_n\|_{W^{\alpha,p}_\Delta(\mathbb{T})} \le M.
\]

\smallskip
\noindent\textbf{Step 1: extraction on each interval component.} Fix 
$\ell \in \{1, \dots, m\}$. The restrictions $(u_n|_{I_\ell})$ form a bounded 
sequence in $W^{\alpha,p}(I_\ell)$ since $\mu_\Delta|_{I_\ell}$ is the 
Lebesgue measure. By the classical fractional Rellich--Kondrachov theorem 
on bounded intervals \cite[Corollary 7.2]{DiNezzaPalatucciValdinoci2012}, this 
sequence is precompact in $L^q(I_\ell)$ for every $q$ in the range stated 
in (i)--(ii). By a diagonal extraction over $\ell = 1, \dots, m$, we obtain 
a subsequence (still denoted $(u_n)$) and limits $v_\ell \in L^q(I_\ell)$ 
such that
\[
u_n|_{I_\ell} \to v_\ell \quad \text{in } L^q(I_\ell), \qquad \ell = 1, \dots, m.
\]

\smallskip
\noindent\textbf{Step 2: extraction on singleton components.} For each 
$j \in \{1, \dots, N\}$, the sequence $(u_n(d_j))_n$ is bounded in $\mathbb{R}$. 
Indeed,
\[
|u_n(d_j)|^p \, \mu_\Delta(\{d_j\}) 
\le \|u_n\|^p_{L^p_\Delta(\mathbb{T})} \le M^p,
\]
hence $|u_n(d_j)| \le M \, \mu_\Delta(\{d_j\})^{-1/p}$. By the Bolzano--Weierstrass 
theorem and a further diagonal extraction, there exist $\xi_j \in \mathbb{R}$ such that
\[
u_n(d_j) \to \xi_j \quad \text{in } \mathbb{R}, \qquad j = 1, \dots, N.
\]

\smallskip
\noindent\textbf{Step 3: assembly of the limit.} Define $u^* : \mathbb{T} \to \mathbb{R}$ by
\[
u^*|_{I_\ell} = v_\ell, \qquad u^*(d_j) = \xi_j.
\]
Then $u^* \in L^q_\Delta(\mathbb{T})$ and
\[
\|u_n - u^*\|^q_{L^q_\Delta(\mathbb{T})} 
= \sum_{\ell=1}^m \|u_n - v_\ell\|^q_{L^q(I_\ell)} 
+ \sum_{j=1}^N |u_n(d_j) - \xi_j|^q \, \mu_\Delta(\{d_j\}).
\]
Both terms tend to $0$ by Steps~1 and~2, hence $u_n \to u^*$ in $L^q_\Delta(\mathbb{T})$.

\smallskip
\noindent\textbf{Step 4: $u^* \in W^{\alpha,p}_\Delta(\mathbb{T})$.} 
By Fatou's lemma applied to the integrand on $\Omega_\mathbb{T}$,
\[
[u^*]^p_{W^{\alpha,p}_\Delta(\mathbb{T})} 
\le \liminf_{n \to \infty} [u_n]^p_{W^{\alpha,p}_\Delta(\mathbb{T})} 
\le M^p.
\]
This concludes the proof.
\end{proof}
\begin{remark}\label{rem:endpoint-noncompact}
The endpoint $q = p^*_\alpha$ is excluded from the compact range
in part~(i). This is sharp: by the classical counterexample on
bounded intervals \cite[Remark~7.3]{DiNezzaPalatucciValdinoci2012},
the embedding $W^{\alpha,p}(I_\ell)\hookrightarrow L^{p^*_\alpha}(I_\ell)$
is continuous but not compact. Since the global space
$W^{\alpha,p}_\Delta(\T)$ restricts to $W^{\alpha,p}(I_\ell)$ on
each interval component (with $\mu_\Delta|_{I_\ell}$ being the
Lebesgue measure), non-compactness at the endpoint propagates to
$W^{\alpha,p}_\Delta(\T)\hookrightarrow L^{p^*_\alpha}_\Delta(\T)$.
\end{remark}
\begin{corollary}[Variational consequence]\label{cor:variational}
Under Assumption~\ref{ass:hybrid-class}, every functional 
$\mathcal{F} : W^{\alpha,p}_\Delta(\mathbb{T}) \to \mathbb{R}$ that is coercive 
with respect to $\|\cdot\|_{W^{\alpha,p}_\Delta(\mathbb{T})}$ and weakly lower 
semicontinuous attains its infimum on every weakly closed bounded subset of 
$W^{\alpha,p}_\Delta(\mathbb{T})$. In particular, the direct method of the 
calculus of variations applies, opening the framework to a wide range of 
nonlocal variational problems on time scales.
\end{corollary}
\subsection{Fractional Hardy and Caffarelli--Kohn--Nirenberg-type
inequalities}
\label{subsec:hardy-ckn}

\begin{theorem}[Fractional Hardy inequality on a bounded time scale]
\label{thm:hardy}
Let $\T\subset\R$ be a bounded hybrid time scale satisfying
\cref{ass:hybrid-class}, and let $x_0$ belong to a nondegenerate interval
component of $\T$. Let $1<p<\infty$, $0<\alpha<1$, and assume
\[
0\le \beta<\alpha,\qquad \beta p<1.
\]
Then there exists $C>0$, depending only on $\T$, $\alpha$, $p$, $\beta$, and
$x_0$, such that
\begin{equation}\label{eq:hardy}
\int_{\T}
\frac{|u(t)-u_{\T}|^p}{|t-x_0|^{\beta p}}\,d\mu_\Delta(t)
\le
C\,\iint_{\Omega_{\T}}
\frac{|u(t)-u(s)|^p}{|t-s|^{1+\alpha p}}
\,d(\mu_\Delta\otimes\mu_\Delta)(t,s)
\end{equation}
for every $u\in W^{\alpha,p}_\Delta(\T)$.
\end{theorem}

\begin{remark}\label{rem:hardy-subcritical-necessary}
The restriction $\beta p<1$ is essential for the formulation with the global
centering $u-u_\T$. Indeed, $u-u_\T$ need not vanish at the singular point
$x_0$. Hence, if $\beta p\ge 1$, the weighted integral may diverge even for
smooth functions on an interval component. Critical Hardy inequalities with
weight $|t-x_0|^{-\alpha p}$ require a boundary condition such as
$u(x_0)=0$, or an equivalent subtraction of the trace at $x_0$.
\end{remark}

\begin{proof}
Set $v:=u-u_{\T}$, so that $v_{\T}=0$ and
$[v]_{W^{\alpha,p}_\Delta(\T)}=[u]_{W^{\alpha,p}_\Delta(\T)}$. We decompose
the left-hand side over the finite family $\mathcal C$ of connected components.

\medskip\noindent
\textbf{Step~1: the interval component containing $x_0$.}
Let $I_0=[a_0,b_0]$ be the nondegenerate interval component containing $x_0$.
On $I_0$, the measure $\mu_\Delta$ coincides with the Lebesgue measure
\cite[Theorem~5.2]{CabadaVivero2006}. Since $\beta p<1$, the weight
$|t-x_0|^{-\beta p}$ is integrable on $I_0$. We write
\[
\int_{I_0}\frac{|v(t)|^p}{|t-x_0|^{\beta p}}\,dt
\le 2^{p-1}\left(
\int_{I_0}\frac{|v(t)-v_{I_0}|^p}{|t-x_0|^{\beta p}}\,dt
+ |v_{I_0}|^p\int_{I_0}|t-x_0|^{-\beta p}\,dt
\right).
\]
The constant term is controlled by Jensen's inequality and \cref{thm:poincare}:
\[
|v_{I_0}|^p\le \mu_\Delta(I_0)^{-1}\int_{I_0}|v|^p\,d\mu_\Delta
\le C\,[u]^p_{W^{\alpha,p}_\Delta(\T)}.
\]
It remains to estimate the weighted oscillation term.

If $\alpha p<1$, the one-dimensional fractional Sobolev embedding on $I_0$
gives $W^{\alpha,p}(I_0)\hookrightarrow L^{p_\alpha^*}(I_0)$, where
$p_\alpha^*=p/(1-\alpha p)$. Choose $s=1/\alpha$, so that
$1/p=1/s+1/p_\alpha^*$. Since $\beta<\alpha$, the weight
$|t-x_0|^{-\beta}$ belongs to $L^s(I_0)$. H\"older's inequality then yields
\[
\int_{I_0}\frac{|v(t)-v_{I_0}|^p}{|t-x_0|^{\beta p}}\,dt
\le C\|v-v_{I_0}\|_{L^{p_\alpha^*}(I_0)}^p
\le C\,[u]^p_{W^{\alpha,p}_\Delta(\T)}.
\]

If $\alpha p=1$, then $W^{\alpha,p}(I_0)$ embeds continuously into
$L^q(I_0)$ for every finite $q\ge p$. Since $\beta<1/p$, choose $q>p$ such
that $1/s:=1/p-1/q>\beta$. Then $|t-x_0|^{-\beta}\in L^s(I_0)$, and the same
H\"older argument gives the desired bound.

If $\alpha p>1$, the Morrey embedding gives
$W^{\alpha,p}(I_0)\hookrightarrow L^\infty(I_0)$. Since $\beta p<1$,
$|t-x_0|^{-\beta}\in L^p(I_0)$, and therefore
\[
\int_{I_0}\frac{|v(t)-v_{I_0}|^p}{|t-x_0|^{\beta p}}\,dt
\le C\|v-v_{I_0}\|_{L^\infty(I_0)}^p
\le C\,[u]^p_{W^{\alpha,p}_\Delta(\T)}.
\]
Thus
\[
\int_{I_0}\frac{|v(t)|^p}{|t-x_0|^{\beta p}}\,dt
\le C\,[u]^p_{W^{\alpha,p}_\Delta(\T)}.
\]

\medskip\noindent
\textbf{Step~2: the remaining components.}
If $C\in\mathcal C$ and $C\ne I_0$, then
$\operatorname{dist}(C,x_0)>0$. Hence $|t-x_0|^{-\beta p}$ is bounded on
$C$, and
\[
\int_C\frac{|v(t)|^p}{|t-x_0|^{\beta p}}\,d\mu_\Delta(t)
\le C_C\int_C |v(t)|^p\,d\mu_\Delta(t).
\]
The right-hand side is controlled by the same intra-component and
component-average argument used in the proof of \cref{thm:poincare}: the
oscillation $v-v_C$ is controlled inside interval components, it vanishes on
singletons, and the averages are controlled by the cross-component interaction
through \cref{lem:discrete-poincare,lem:cross-component-bounds}. Consequently,
\[
\sum_{C\in\mathcal C,\ C\ne I_0}
\int_C\frac{|v(t)|^p}{|t-x_0|^{\beta p}}\,d\mu_\Delta(t)
\le C\,[u]^p_{W^{\alpha,p}_\Delta(\T)}.
\]

Combining Step~1 and Step~2 proves \eqref{eq:hardy}.
\end{proof}

\begin{remark}[Critical Hardy exponent]\label{rem:critical-hardy}
When $\alpha p>1$, the Morrey embedding
$W^{\alpha,p}$ on the interval component containing $t_0$ embeds into a H\"older space, so pointwise values are well defined there. In this regime, a
stronger Hardy inequality with the \emph{critical} weight
$|t-t_0|^{-\alpha p}$ holds for functions satisfying $u(t_0)=0$; see
\cite{Dyda2004} and \cite[Theorem~1.1]{FrankSeiringer2008} for the
classical one-dimensional result. The extension to hybrid time scales under
\cref{ass:hybrid-class} follows by the same component decomposition used
above.
\end{remark}

We now combine the Hardy inequality with a general interpolation argument to
obtain a Caffarelli--Kohn--Nirenberg-type estimate. The following proposition
is stated abstractly: it holds on any finite measure space and reduces to
H\"older's inequality applied to a suitable product decomposition. Its
interest in the present setting lies in its combination with
\cref{thm:hardy,thm:embedding}.

\begin{proposition}[A Caffarelli--Kohn--Nirenberg-type interpolation on a
bounded time scale]
\label{prop:CKN-type}
Let $\T\subset\R$ be a bounded time scale satisfying
\cref{ass:hybrid-class}, let $x_0\in\T$, let $1<p<\infty$,
$1\le q<\infty$, and let $0\le \theta\le1$. Define $r$ by
\[
\frac1r=\frac{\theta}{p}+\frac{1-\theta}{q}.
\]
Let $\beta\ge0$, and set $b:=\theta\beta$. Assume that there exists
$C_H>0$ such that
\[
\bigl\||t-x_0|^{-\beta}(u-u_{\T})\bigr\|_{L^p_\Delta(\T)}
\le
C_H\,[u]_{W^{\alpha,p}_\Delta(\T)}
\qquad\text{for all }u\in W^{\alpha,p}_\Delta(\T).
\]
Then
\begin{equation}\label{eq:ckn-abstract}
\bigl\||t-x_0|^{-b}(u-u_{\T})\bigr\|_{L^r_\Delta(\T)}
\le
C_H^{\theta}\,[u]_{W^{\alpha,p}_\Delta(\T)}^{\theta}
\,\|u-u_{\T}\|_{L^q_\Delta(\T)}^{1-\theta}
\end{equation}
for all $u\in W^{\alpha,p}_\Delta(\T)$.
\end{proposition}

\begin{proof}
Set $v:=u-u_{\T}$. Then
\[
|t-x_0|^{-br}|v|^r
=
\Bigl(|t-x_0|^{-\beta p}|v|^p\Bigr)^{\frac{\theta r}{p}}
\Bigl(|v|^q\Bigr)^{\frac{(1-\theta)r}{q}}.
\]
Since
\[
\frac{\theta r}{p}+\frac{(1-\theta)r}{q}=1,
\]
H\"older's inequality with conjugate exponents $p/(\theta r)$ and
$q/((1-\theta)r)$ gives
\[
\int_{\T}|t-x_0|^{-br}|v|^r\,d\mu_\Delta
\le
\left(
\int_{\T}|t-x_0|^{-\beta p}|v|^p\,d\mu_\Delta
\right)^{\frac{\theta r}{p}}
\left(
\int_{\T}|v|^q\,d\mu_\Delta
\right)^{\frac{(1-\theta)r}{q}}.
\]
Taking the $r$-th root and applying the assumed Hardy estimate yields
\eqref{eq:ckn-abstract}.
\end{proof}
\begin{remark}\label{rem:ckn-admissible-q}
When applying Proposition~\ref{prop:CKN-type} in combination with
the Sobolev embedding of Theorem~\ref{thm:embedding} to control
$\|u-u_\T\|_{L^q_\Delta(\T)}$, the exponent $q$ must lie in the
admissible range guaranteed by that theorem:
\begin{itemize}
\item $p \le q \le p^*_\alpha := p/(1-\alpha p)$
      \quad if $\alpha p < 1$,
\item $p \le q < \infty$
      \quad if $\alpha p = 1$
      (Trudinger-type embedding),
\item $p \le q \le \infty$
      \quad if $\alpha p > 1$
      (Morrey embedding).
\end{itemize}
Outside this range, the factor $\|u-u_\T\|^{1-\theta}_{L^q_\Delta(\T)}$
on the right-hand side of \eqref{eq:ckn-abstract} need not be
controlled by the Gagliardo seminorm alone.
\end{remark}
Combining \cref{prop:CKN-type} with the Hardy inequality
(\cref{thm:hardy}) and the fractional Sobolev embedding
(\cref{thm:embedding}) yields a self-contained
Caffarelli--Kohn--Nirenberg-type inequality in which the right-hand side is
controlled entirely by the Gagliardo seminorm.

\begin{corollary}[Caffarelli--Kohn--Nirenberg-type inequality]
\label{cor:ckn}
Under \cref{ass:hybrid-class}, let $x_0$ belong to a nondegenerate interval
component of $\T$. Let $1<p<\infty$, $0<\alpha<1$, $0\le\beta<\alpha$,
$\beta p<1$, $0\le\theta\le 1$, and assume moreover that $\alpha p<1$ and
\[
1\le q\le p_\alpha^*:=\frac{p}{1-\alpha p}.
\]
Define $r$ and $b$ as in \cref{prop:CKN-type}. Then there exists $C>0$ such
that
\begin{equation}\label{eq:ckn}
\bigl\||t-x_0|^{-b}(u-u_{\T})\bigr\|_{L^r_\Delta(\T)}
\le
C\,[u]_{W^{\alpha,p}_\Delta(\T)}
\end{equation}
for every $u\in W^{\alpha,p}_\Delta(\T)$.
\end{corollary}

\begin{proof}
By \cref{thm:hardy},
\[
\bigl\||t-x_0|^{-\beta}(u-u_\T)\bigr\|_{L^p_\Delta(\T)}
\le C_H [u]_{W^{\alpha,p}_\Delta(\T)}.
\]
Since $q$ lies in the fractional Sobolev range, \cref{thm:embedding} gives
\[
\|u-u_{\T}\|_{L^q_\Delta(\T)}
\le C_E\,[u]_{W^{\alpha,p}_\Delta(\T)}.
\]
Substituting these two estimates into \eqref{eq:ckn-abstract} yields
\[
\bigl\||t-x_0|^{-b}(u-u_\T)\bigr\|_{L^r_\Delta(\T)}
\le C_H^\theta C_E^{1-\theta}[u]_{W^{\alpha,p}_\Delta(\T)},
\]
which is \eqref{eq:ckn}.
\end{proof}

To illustrate Theorem~\ref{thm:hardy} and
Proposition~\ref{prop:CKN-type} on a concrete hybrid time scale,
consider the following three-component example, which combines
two interval components with one isolated point.

\begin{example}\label{ex:hardy-ckn}
Let $\T=[0,1]\cup\{2\}\cup[3,4]$, $\alpha=\frac{1}{2}$, $p=2$,
and $\beta=\frac{1}{4}$.
This time scale satisfies Assumption~\ref{ass:hybrid-class} with
$\delta_0=1$, $\mathcal{C}=\{[0,1],\,\{2\},\,[3,4]\}$.
The Lebesgue measure gives $\mu_\Delta([0,1])=\mu_\Delta([3,4])=1$.
For the isolated point $2$, since
$\sigma_\T(2):=\inf\bigl(\T\cap(2,\infty)\bigr)=3$,
we have
\[
\mu_\Delta(\{2\})=\sigma_\T(2)-2=3-2=1>0,
\]
so all three components contribute to the measure.

Since $\alpha p=1$, the Morrey--Trudinger embedding gives
$W^{1/2,2}_\Delta(\T)\hookrightarrow L^q_\Delta(\T)$
for every $q<\infty$ (see Remark~\ref{rem:morrey-trudinger}).
The conditions $\beta<\alpha$ and $\beta p=\frac{1}{2}<1$ place the example in the subcritical weighted regime of Theorem~\ref{thm:hardy}, so the Hardy inequality reads
\begin{align*}
&\int_0^1\frac{|u(t)-u_\T|^2}{t^{1/2}}\,dt
+\frac{|u(2)-u_\T|^2}{2^{1/2}}\cdot\mu_\Delta(\{2\})
+\int_3^4\frac{|u(t)-u_\T|^2}{t^{1/2}}\,dt\\
&\qquad\le\; C_H\,[u]^2_{W^{1/2,2}_\Delta(\T)},
\end{align*}
i.e.,
\[
\int_0^1\frac{|u(t)-u_\T|^2}{t^{1/2}}\,dt
+\frac{|u(2)-u_\T|^2}{\sqrt{2}}
+\int_3^4\frac{|u(t)-u_\T|^2}{t^{1/2}}\,dt
\le C_H\,[u]^2_{W^{1/2,2}_\Delta(\T)}.
\]
The three terms on the left-hand side reflect the three distinct
contributions: the singular weight $t^{-1/2}$ on $[0,1]$
(classical one-dimensional Hardy behaviour near $0$),
a single weighted point evaluation at the isolated node $2$
(controlled by the cross-component interaction energy),
and the regular weight $t^{-1/2}$ on $[3,4]$
(bounded away from $0$, hence dominated by
$\|u-u_\T\|^2_{L^2_\Delta([3,4])}$).

Applying Proposition~\ref{prop:CKN-type} with $\theta=\frac{1}{2}$,
$q=4$ (admissible in the critical embedding regime), and $r$ defined by
$1/r=\theta/p+(1-\theta)/q=1/4+1/8=3/8$, i.e.\ $r=8/3$,
the Caffarelli--Kohn--Nirenberg-type inequality gives
\[
\bigl\||t|^{-b}(u-u_\T)\bigr\|_{L^{8/3}_\Delta(\T)}
\le C_H^{1/2}\,C_E^{1/2}\,[u]_{W^{1/2,2}_\Delta(\T)},
\]
where $b=\theta\beta=1/8$ and $C_E$ is the embedding constant
from the Sobolev theorem.
\end{example}


\section{Conclusion}
\label{sec:conclusion}

We have developed a systematic Gagliardo-type formulation of fractional Sobolev spaces on arbitrary time scales, built directly from the Lebesgue $\Delta$-measure and the off-diagonal interaction domain. The construction provides a nonlocal notion of fractional regularity that is structurally distinct from the derivative-based approaches previously studied in the literature, and that is especially natural on hybrid time scales, where it combines continuous, discrete, and mixed interactions within a single energy.

The results are organized along three axes. On arbitrary time scales, we established the basic functional structure of the new spaces, including completeness, reflexivity, and the Hilbert case under an equivalent norm, together with a sharp nontriviality criterion on bounded time scales with finitely many connected components. On the interface with the derivative-based theory, we delimited precisely the regime in which the two viewpoints coincide: on continuous intervals, in the supercritical regime, the Gagliardo norm and the bilateral Riemann--Liouville norm are equivalent on the subspace of functions with vanishing boundary trace; on hybrid time scales, an explicit obstruction rules out any analogous equivalence, identifying the mixed continuous--discrete interactions as the structural cause. On bounded hybrid time scales with positively separated components, we proved a first set of geometric estimates: a Poincar\'e-type inequality with quantified dependence on the component gap, a fractional Sobolev embedding, and fractional Hardy and Caffarelli--Kohn--Nirenberg-type inequalities for subcritical weights.

Several natural questions are left for further work. The most immediate is the analogue of the supercritical Riemann--Liouville equivalence in the subcritical regime, which requires either a trace theory on the Gagliardo space or a density-based definition of the boundary-zero subspace. Compactness in the form of a Rellich--Kondrachov theorem, variable-order extensions, and a finer comparison with the weighted framework of Tan, Zhou, and Wang are equally natural directions. Beyond these, the variational use of the present framework for nonlocal problems on hybrid time scales, where the off-diagonal kernel encodes information that no differential operator can capture, appears as a promising application of the theory.


\section*{Declarations}

\subsection*{Competing interests} The authors declare that they have no known competing financial interests or personal relationships that could have appeared to influence the work reported in this paper.

\subsection*{Funding sources}
D.~F.~M.~Torres is supported by CIDMA under the Portuguese Foundation for Science and Technology (FCT), grant UID/04106/2025. The other authors did not receive any specific grant from funding agencies in the public, commercial, or not-for-profit sectors.

\subsection*{Declaration of Generative AI and AI-assisted technologies in the writing process}
During the preparation of this work, the authors used AI to improve language, readability, and presentation of the manuscript. The authors edited the content and take full responsibility for the final version of the manuscript.



\end{document}